\documentclass[a4paper,11pt]{article}
\usepackage{pdfsync}

\usepackage[plainpages=false]{hyperref}
\usepackage{amsfonts,amsmath,amssymb,amsthm,enumerate}
\usepackage{latexsym,lscape,rawfonts,mathrsfs,mathtools}
\usepackage{enumitem,pifont}
\usepackage{cases}
\usepackage{authblk}

\usepackage[all]{xy}
\usepackage{eufrak}
\usepackage{makeidx}
\usepackage{graphicx,psfrag}
\usepackage{pstool}
\usepackage{tikz-cd}
\usepackage{float}
\usepackage{array}
\usepackage{diagbox}

\usepackage{setspace}

\usepackage[titletoc,title]{appendix}

\usepackage{txfonts}

\newcommand{\CC}{\ensuremath{\mathcal{C}}}

\newcommand{\nnn}{\ensuremath{\mathscr{N}}}
\newcommand{\ccc}{\ensuremath{\mathscr{C}}}

\newcommand{\MM}{\ensuremath{\mathcal{M}}}

\newcommand{\N}{\ensuremath{\mathbb{N}}}

\newcommand{\R}{\ensuremath{\mathbb{R}}}

\newcommand{\VV}{\ensuremath{\mathcal{V}}}

\newcommand{\bballY}{there exists point in $\pball(Y_b,4r_b) \bigcap \MM$ which is $(j', \eta, r_b)$-cylindrical for some $j'>j$}

\newcommand{\MS}{\ensuremath{\mathcal{S}}}

\newcommand{\ba}{\begin{align*}}
\newcommand{\ea}{\end{align*}}
\newcommand{\na}{\nabla}

\newcommand{\lc}{\left(}
\newcommand{\rc}{\right)}

\newcommand{\ep}{\epsilon}

\newcommand{\tf}{\mathfrak{t}}
\newcommand{\ef}{\mathfrak{e}}

\newcommand{\pball}{\mathfrak{B}}
\newcommand{\spacetime}{\mathbb{R}^{n+1,1}}
\newcommand{\espace}{\mathbb{R}^{n+1}}

\newcommand{\vol}{\mathrm{Vol}}

\makeatletter
\newcommand*\owedge{\mathpalette\@owedge\relax}
\newcommand*\@owedge[1]{%
\mathbin{%
\ooalign{%
$#1\m@th\bigcirc$\cr
\hidewidth$#1\m@th\wedge$\hidewidth\cr
}%
}%
}
\makeatother

\makeatletter
\def\ExtendSymbol#1#2#3#4#5{\ext@arrow 0099{\arrowfill@#1#2#3}{#4}{#5}}

\makeatother

\makeatletter
\def\ExtendSymbol#1#2#3#4#5{\ext@arrow 0099{\arrowfill@#1#2#3}{#4}{#5}}

\makeatother

\def\XXint#1#2#3{{\setbox0=\hbox{$#1{#2#3}{\int}$ }
\vcenter{\hbox{$#2#3$ }}\kern-.55\wd0}}

\numberwithin{equation}{section}

\newtheorem{thm}{Theorem}[section]
\newtheorem{cor}[thm]{Corollary}
\newtheorem{prop}[thm]{Proposition}
\newtheorem{lem}[thm]{Lemma}

\newtheorem{defn}[thm]{Definition}
\newtheorem{claim}[thm]{Claim}

\newtheorem{exmp}[thm]{Example}

\title{Volume estimates for the singular sets of mean curvature flows}
\author{Hanbing Fang \quad and \quad Yu Li} 
\date{\today}

\begin{document}
\maketitle

\begin{abstract}
In this paper, we establish uniform and sharp volume estimates for the singular set and the quantitative singular strata of mean curvature flows starting from a smooth, closed, mean-convex hypersurface in $\R^{n+1}$.
\end{abstract}


\section{Introduction}

A one-parameter family $\{M_t\}_{t \in I}$ of hypersurfaces in $\R^{n+1}$ evolves by mean curvature flow if
\begin{equation*}
(\partial_t x)^{\perp}=\mathbf H(x),
\end{equation*}
where $\mathbf{H}=-H \mathbf{n}$ is the mean curvature vector, $\mathbf n$ is the outward unit normal, $v^{\perp}$ denotes the normal component of a vector $v$, and the mean curvature $H$ is given by
\begin{equation*}
H=\mathrm{div}(\mathbf{n}).
\end{equation*}

Mean curvature flow is the gradient flow of the area functional. The simplest example of a non-static mean curvature flow is the family of shrinking spheres $\{S^n(\sqrt{2n|t|})\}_{t<0}$. In this case, the flow is smooth except at the origin at time $t=0$, where it becomes extinct. More generally, Huisken \cite{huisken1984flow} proved that if the initial hypersurface $M_0$ is closed and convex, then $M_t$ becomes increasingly spherical as it contracts to a point.

Recall that a mean curvature flow $\{M_t\}_{t \ge 0}$ from a closed hypersurface $M_0 \subset \R^{n+1}$ must develop singularities in finite time. In studying mean curvature flow, a useful weak solution concept is the level set flow (see \cite{evans1991motion} and \cite{chen1991uniqueness}): Suppose $f$ is a function satisfying the equation
\begin{equation} \label{eq:level}
\partial_t f=|\na f|\,\mathrm{div}\lc \frac{\na f}{|\na f|} \rc,
\end{equation}
then the level sets $M_t=\{x \in \R^{n+1} \mid f(x, t)=0\}$ are hypersurfaces evolving under mean curvature flow. In general, Evans-Spruck \cite{evans1991motion} and Chen-Giga-Goto \cite{chen1991uniqueness} established the existence and uniqueness of viscosity solutions to \eqref{eq:level}, providing a weak formulation of mean curvature flow.

A fundamental question is understanding the formation of singularities for the level set flow before its extinction. A special case of interest is the level set flow from a smooth, closed, mean-convex hypersurface, meaning that the mean curvature of the initial hypersurface is nonnegative. The mean-convexity assumption leads to many significant results, as shown in \cite{sinestrari1999convexity} \cite{huisken1999mean} \cite{brendle2016mean} \cite{white2000size} \cite{white2003nature} \cite{white2015subsequent} \cite{haslhofer2017mean1} \cite{haslhofer2017mean2} \cite{colding2016singular} \cite{colding2016differentiability}, among others. In this context, one can express the level set function $f(x,t)=u(x)-t$ (see \cite[Section 7.3]{evans1991motion}), so that \eqref{eq:level} becomes
\begin{align} \label{eq:level2}
-1=|\nabla u|\ \mathrm{div}\lc \frac{\nabla u}{|\nabla u|} \rc,
\end{align}
where $u$ is referred to as the \emph{arrival time}. The level set $M_t=\{u(x)=t\}$ encloses a mean-convex domain. The family $\{M_t\}_{t \ge 0}$ moves inward, sweeping out the domain $\Omega$ enclosed by $M_0$. In \cite{colding2016differentiability}, Colding and Minicozzi proved that $u$ is twice differentiable, and thus the equation \eqref{eq:level2} holds in the classical sense. Furthermore, it is shown in \cite[Corollary 3.3]{white2000size} that the level set flow is nonfattening, meaning the level sets have no interior.

The singular set of the level set flow is given by
\begin{equation*}
\mathcal S=\{(x,u(x)) \mid \na u(x)=0\}.
\end{equation*}
In \cite{white2000size}, White established the significant result that $\mathcal{S}$ has parabolic Hausdorff dimension of at most $n-1$. Later, in \cite{white2003nature} and \cite{white2015subsequent}, White demonstrated that all tangent flows are generalized cylinders of the form $\{\mathbb{R}^k\times S^{n-k}(\sqrt{2(n-k)|t|})\}_{t \le 0}$ with multiplicity one. Furthermore, in \cite{colding2015uniqueness}, Colding and Minicozzi proved that the axis of these generalized cylinders are unique by establishing a discrete {\L}ojasiewicz inequality. 

The singular set $\mathcal S$ admits the following stratification:
\begin{equation*}
\mathcal{S}^0\subset\mathcal{S}^1\subset\cdots \subset \mathcal{S}^{n-1}=\mathcal{S},
\end{equation*}
where $\mathcal S^k$ consists of all singularities for which the tangent flow splits off an Euclidean factor of dimension at most $k$. Additionally, White proved in \cite{white1997stratification} that the parabolic Hausdorff dimension of $\mathcal S^k$ is at most $k$. Colding and Minicozzi further showed in \cite[Theorem 4.6]{colding2016singular} that each stratum $\mathcal S^k$ is contained in countably many of $k$-dimensional $C^1$-submanifolds, along with a set of lower dimension.

To further analyze the singular set, we define the following quantitative singular stratum, analogous to those in \cite{cheeger2013lower} and \cite{cheeger2013quantitative}.

\begin{defn} \label{def:sin1}
For $j\in\{0,1,\ldots,n-1\}$ and $r \in (0,1)$, the $j$-th quantitative singular stratum of the level set flow $\MM=\{M_t\}_{t \ge 0}$ is define as
\begin{equation*}
\mathcal{S}^{j}_{\epsilon,r}:=\{X\in \mathcal{M}\mid X \text{ is not } (j',\ep,s)\text{-cylindrical} \text{ for any } s\in [r,1) \text{ and any } j'>j\}.
\end{equation*}
\end{defn}
Here, a spacetime point $X \in \MM$ is $(j, \ep, r)$-cylindrical if, at scale $r$, the flow in a ball of size $\ep^{-1}$ can be written as a graph over a cylinder $\mathbb{R}^j\times S^{n-j}(\sqrt{2(n-j)})$ with $C^{[\ep^{-1}]}$-norm smaller than $\ep$; see Definition \ref{deficylindrical} for the precise formulation. It can be shown (see Proposition \ref{prop:decom1}) that
\begin{equation*}
\MS^j=\bigcup_{\ep>0}\bigcap_{r\in (0,1)}\MS^{j}_{\ep,r}.
\end{equation*}

In \cite{cheeger2013quantitative}, Cheeger, Haslhofer, and Naber introduced a quantitative singular stratum for general Brakke flows of any codimension, establishing a Minkowski-type volume estimate (see \cite[Theorem 1.14]{cheeger2013quantitative}). Our first main result provides a uniform and sharp volume estimate for the $j$-th quantitative singular stratum $\mathcal{S}^{j}_{\epsilon,r}$, where the bound depends essentially on the entropy (see equation \eqref{eq:entro}) and the noncollapsing constant (see Definition \ref{def:alpha}) of the initial hypersurface.

\begin{thm}\label{volumeestimate1}
Let $\mathcal{M}=\{M_t\}_{t \ge 0}$ be a level set flow starting from a closed, smooth, mean-convex, $\alpha$-noncollapsed hypersurface $M_0 \subset \R^{n+1}$ with $\lambda(M_0) \le \Lambda$. Then, for any $\ep>0$, there exists a constant $C=C(n, \alpha, \Lambda, \ep)>0$ such that for any $X \in \R^{n+1,1}$ with $\tf(X)>2$,
\begin{equation} \label{eq:main11}
\vol \lc \pball_r(\mathcal{S}^{j}_{\epsilon,r}) \bigcap \pball(X,1) \rc\leq Cr^{n+3-j},
\end{equation}
for $j\in\{0,1,\ldots,n-1\}$ and any $r \in (0,1)$. In particular, for $\MS^j_{\ep}:=\bigcap_{r\in (0,1)}\MS^{j}_{\ep,r}$, we have
\begin{equation} \label{eq:main21}
\mathscr{H}_P^j \lc \mathcal{S}_{\epsilon}^{j} \bigcap \pball(X, 1) \rc \le C.
\end{equation}
\end{thm}

Here, $\vol(\cdot)$ denotes the parabolic volume, and $\pball(\cdot)$ and $\mathscr{H}_P^j$ denote balls and the $j$-dimensional parabolic Hausdorff measure with respect to the parabolic distance (see \eqref{parabolicdist}), respectively. For further notation, readers may refer to subsection \ref{sub:nota} for precise definitions.

It was proved by Colding and Minicozzi \cite[Theorem 4.6]{colding2016singular} that the top stratum $\MS \setminus \MS^{n-2}$ is contained in finitely many $(n-1)$-dimensional $C^1$-submanifolds. In particular, this implies that $\MS$ has locally finite $(n-1)$-dimensional parabolic Hausdorff measure. Since $\MS^{n-1}_{\ep}=\MS$ if $\ep \le \ep(n, \alpha, \Lambda)$ (see Proposition \ref{prop:top}), we obtain the following result as a corollary of Theorem \ref{volumeestimate1}.

\begin{cor}
Let $\mathcal{M}=\{M_t\}_{t \ge 0}$ be a level set flow starting from a closed, smooth, mean-convex, $\alpha$-noncollapsed hypersurface $M_0 \subset \R^{n+1}$ with $\lambda(M_0) \le \Lambda$. Then, there exists a constant $C=C(n, \alpha, \Lambda)>0$ such that for any $X \in \R^{n+1,1}$ with $\tf(X)>2$,
\begin{equation*}
\mathscr{H}_P^{n-1}\lc \MS\bigcap \pball(X,1)\rc\leq C.
\end{equation*}
\end{cor}

If the initial hypersurface $M_0$ is strictly $k$-convex, in the sense that $\lambda_1+\cdots+\lambda_k>0$, where $\lambda_1 \le \cdots \le \lambda_n$ are the principal curvatures, then we can assume 
\begin{equation}\label{uniformkconvex}
\lambda_1+\cdots+\lambda_k\geq \beta H>0
\end{equation}
for some $\beta>0$. It can be shown (see \cite{cheeger2013quantitative} and \cite{haslhofer2017mean1}) that condition \eqref{uniformkconvex} is preserved and holds at any smooth point under the level set flow. Furthermore, this condition restricts all tangent flows at a singular point to be of the form of $\{\R^j\times S^{n-j}(\sqrt{2(n-j)|t|})\}_{t \le 0}$ for $j\in\{0,1,\ldots, k-1\}$, implying that $\MS=\MS^{k-1}$. Similar to Proposition \ref{prop:top}, if $\ep \le \ep(n, \alpha, \beta, \Lambda)$, then $\MS^{k-1}_{\ep}=\MS$, and we obtain the following result:

\begin{cor}
Let $\mathcal{M}=\{M_t\}_{t \ge 0}$ be a level set flow starting from a closed, smooth, $\alpha$-noncollapsed hypersurface $M_0 \subset \R^{n+1}$ satisfying condition \eqref{uniformkconvex} and $\lambda(M_0) \le \Lambda$. Then, there exists a constant $C=C(n, \alpha, \beta, \Lambda)>0$ such that for any $X \in \R^{n+1,1}$ with $\tf(X)>2$,
\begin{equation*}
\mathscr{H}_P^{k-1} \lc \MS\bigcap \pball(X,1) \rc \leq C.
\end{equation*}
\end{cor}

The proof of Theorem \ref{volumeestimate1} is based on several ingredients. 

The first is a quantitative version (see Theorem \ref{quantitativeuniq}) of the uniqueness of cylindrical tangent flows, as established by Colding and Minicozzi \cite{colding2015uniqueness}. Roughly speaking, this result states that if the Gaussian density ratio $\Theta(X, \tau)$ for a spacetime point $X \in \MM$ is nearly constant at two different scales $\tau_1$ and $\tau_2$, then $X$ is $(j, \ep, r)$-cylindrical for any $r \in [\sqrt \tau_1, \sqrt \tau_2]$, with the same underlying cylinder $\R^j\times S^{n-j}(\sqrt{2(n-j)})$.

Next, we define the concept of a neck region $\nnn$ (see Definition \ref{defiofneckregion}). The notion of a neck region first appeared in \cite{jiang20212} in the proof of the $(n-4)$-finiteness conjecture and was later used to establish the rectifiability of singular sets in noncollapsed limit spaces with Ricci curvature bounded below (see \cite{Cheeger2018RectifiabilityOS}). With the quantitative uniqueness, we obtain a structural characterization of the center $\ccc$ of the neck region. More precisely, it can be shown (see Theorem \ref{neckstructurethm}) that $\ccc$ is a compact Lipschitz graph over a $j$-plane. 

A crucial step is to establish the neck decomposition theorem (see Theorem \ref{neckdecomposition}), which also holds in the context of Ricci limit spaces (cf. \cite[Theorem 2.12]{Cheeger2018RectifiabilityOS}). In essence, we obtain the decomposition
\begin{equation} \label{eq:cov1}
\pball(X, r) \bigcap \MM \subset \bigcup_a (\nnn_a\bigcup \ccc_a) \bigcup_b \pball(Y_b,r_b)\bigcup \MS^{j-1}
\end{equation}
with the following content estimate:
\begin{equation}\label{eq:con1}
\sum_a r_a^{j}+\sum_b r_b^{j}+\mathscr H_P^j(\bigcup_a \ccc_a) \leq C r^j.
\end{equation}
Here, each $\nnn_a \subset \pball(Y_a, 2r_a)$ is a neck region with center $\ccc_a$, each $\pball(Y_b,r_b)$ approximately splits off an $\R^{j'}$ for some $j'>j$ at scale $r_b$, and $\MS^{j-1}$ has lower dimension. We will prove (see Corollary \ref{cor:quali}) that $\MS^{j} \setminus \MS^{j-1}$ is essentially contained in the union of centers $\ccc_a$, thereby recovering the following important structure theorem, originally proved by Colding and Minicozzi \cite{colding2016singular}:

\begin{thm}\label{thm:quali}
Let $\mathcal{M}=\{M_t\}_{t \ge 0}$ be a level set flow starting from a closed, smooth, mean-convex hypersurface. Then,
\begin{enumerate}[label=(\alph*)]
\item $\MS^{n-1}\setminus \MS^{n-2}$ is contained in finite compact $(n-1)$-dimensional Lipschitz graphs.
\item For $j \le n-2$, $\MS^{j}\setminus \MS^{j-1}$ is contained in countable compact $j$-dimensional Lipschitz graphs.
\end{enumerate}
\end{thm}

The proof of the neck decomposition theorem relies on a refined covering argument as in \cite{Cheeger2018RectifiabilityOS}. By dividing all parabolic balls into different types (see Definition \ref{allkindsofballs}), we achieve the decomposition \eqref{eq:cov1} with the content estimate \eqref{eq:con1} via an iterative covering argument (see Proposition \ref{decompositioncball}, Proposition \ref{decompositiondball} and Proposition \ref{inductivedecomposition}). The proof of Theorem \ref{volumeestimate1} is based on the neck decomposition theorem and the content estimate \eqref{eq:con1}.

The organization of this paper is as follows. In Section 2, we review fundamental concepts and results related to mean curvature flows, focusing on level set flows starting from smooth, closed, mean-convex hypersurfaces. In Section 3, we establish the quantitative uniqueness of cylindrical singularities. Section 4, which contains the primary technical details, formulates the definition of a neck region and proves its structural properties. This section also includes the proof of the neck decomposition theorem. Finally, in Section 5, we conclude with the proof of our main theorem.
\\
\\
\\
\textbf{Acknowledgements}: Hanbing Fang would like to thank his advisor, Prof. Xiuxiong Chen, for his encouragement and support. Hanbing Fang is supported by the Simons Foundation. Yu Li is supported by YSBR-001, NSFC-12201597 and research funds from the University of Science and Technology of China and the Chinese Academy of Sciences. 

\section{Preliminaries}

In this section, we review some definitions, notations, and prior results in mean curvature flow that will be frequently used throughout the paper.

\subsection{Notations and conventions} \label{sub:nota}

We define the spacetime $\spacetime = \mathbb{R}^{n+1} \times \mathbb{R}$. Throughout this paper, uppercase letters such as $X, Y, Z$, etc., denote points in $\spacetime$, while lowercase letters $x, y, z$, etc., represent points in $\espace$. The origin in space is denoted by $\vec{0}^{n+1}$, and the origin in spacetime is represented by $\vec{0}^{n+1,1}$. For any $X \in \spacetime$, the space component and time component are denoted by $\ef(X)$ and $\tf(X)$, respectively.

The parabolic distance on $\spacetime$ is defined as follows: for $X,Y \in \spacetime$,
\begin{equation}\label{parabolicdist}
\|X-Y\| := \max\{ |\ef(X)-\ef(Y)|,|\tf(X)-\tf(X)|^{\frac 1 2}\},
\end{equation}
where $|\cdot|$ represents the standard Euclidean distance in $\R^{n+1}$. Additionally, we denote the parabolic ball centered at $X$ with radius $r$ by $\pball(X,r)=B(\ef(X),r)\times (\tf(X)-r^2,\tf(X)+r^2)$.

The $k$-dimensional Hausdorff measure on $\R^{n+1}$ is denoted by $\mathscr{H}^k$, and the $k$-dimensional parabolic Hausdorff measure on $\spacetime$, denoted by $\mathscr{H}^k_P$, is defined with respect to the distance given in \eqref{parabolicdist}. In particular, we use $\vol(\cdot)$ to denote $\mathscr{H}^{n+3}_P$.

We will also use $\Psi(\ep)$ to denote a quantity that satisfies $\Psi(\ep) \to 0$ as $\ep \to 0$, assuming all other related quantities are fixed. Note that $\Psi(\ep)$ may vary from line to line.

\subsection{Integral Brakke flows}

The Brakke flow is an important notion of the weak mean curvature flow (cf. \cite{brakke2015motion}\cite{Ilmanen1994EllipticRA}), whose definition we recall.

\begin{defn}[Integral Brakke flow]
An \textbf{$n$-dimensional integral Brakke flow} $\mathcal{M}$ in $\R^{n+1}$ is a one-parameter family of Radon measures $\mathcal M=(\mu_t)_{t \in I}$ over an interval $I \subset \R$ so that:
\begin{enumerate}[label=(\alph*)]
\item For almost every $t \in I$, there exists an integer rectifiable $n$-varifold $V(t)$ with $\mu_t=\mu_{V(t)}$ so that $V(t)$ has locally bounded first variation, and its mean curvature $\mathbf{H}$ is orthogonal to $\mathrm{Tan}(V(t),\cdot)$ almost everywhere.

\item For any $[a,b] \subset I$ and compact set $K \subset \R^{n+1}$
\begin{equation*}
\int_a^b \int_K (1+|\mathbf{H}|^2) \,d\mu_t dt< \infty.
\end{equation*}

\item If $[a,b] \subset I$ and $\varphi \in C_c^1(\R^{n+1} \times [a,b],[0,\infty))$, then
\begin{equation*}
\int \varphi(\cdot, b) \,d\mu_b-\int \varphi(\cdot, a) \,d\mu_a \le \int_a^b \int -|\mathbf{H}|^2 \varphi+\mathbf{H} \cdot \na \varphi+\frac{\partial}{\partial t} \varphi \,d\mu_t dt.
\end{equation*}
\end{enumerate}
\end{defn}

For an integer $n$-rectifiable Radon measure $\mu$ on $\R^{n+1}$, we set
\begin{align*}
F(\mu):= (4\pi)^{-\frac{n}{2}}\int e^{-\frac{|x|^2}{4}}d\mu(x).
\end{align*}
Following \cite{colding2012generic}, the entropy of $\mu$ is defined as
\begin{align*}
\lambda (\mu) := \sup_{(y,\rho) \in \R^{n+1} \times \R_+} F(\mu^{y,\rho}) \quad \text{where} \quad \mu^{y,\rho}(U)=\rho^{n} \mu(\rho^{-1} U+y)
\end{align*}
denotes the rescaled measure. For an integral Brakke flow $\mathcal{M}=(\mu_t)_{t \ge 0}$, we set 
\begin{equation} \label{eq:entro}
\lambda(\mathcal{M})=\sup_{t \ge 0}\lambda(\mu_t).
\end{equation}

The following celebrated monotonicity formula of Huisken holds (cf. \cite{huisken1984flow}\cite[Lemma 7]{Il95}).

\begin{prop}\label{prop:mono}
Let $\mathcal M=(\mu_t)_{t \ge 0}$ be an integral Brakke flow on $\R^{n+1}$ with $\lambda(\mu_{0})<\infty$. Then for any $Y=(y,s) \in \R^{n+1} \times (0,\infty)$ and all $0 \le t_1 \le t_2<s$,
\begin{align*}
&\int \Phi_{Y}(x,t_2) \,d\mu_{t_2}(x)-\int \Phi_{Y}(x,t_2)\,d\mu_{t_1}(x)\le -\int_{t_1}^{t_2} \int \left|\mathbf{H}(x,t)+\frac{(x-y)^{\perp}}{2(s-t)} \right|^2\Phi_{Y}(x,t)\,d\mu_t(x) dt,
\end{align*}
where
\begin{align*}
\Phi_{Y}(x,t)=(4\pi(s-t))^{-n/2}e^{-\frac{|x-y|^2}{4(s-t)}}.
\end{align*}
In particular, $\lambda(\mu_t)$ is non-increasing for $t \ge 0$ and hence $\lambda(\mathcal M)=\lambda(\mu_{0})$. 
\end{prop}

For $X=(x,t)\in \mathbb{R}^{n+1,1}$ with $t>0$ and $\tau>0$, we recall the \textbf{Gaussian density ratio} is defined by
\begin{equation*}
\Theta(\mathcal{M}, X,\tau):=\int \Phi_X(y,t-\tau)\,d\mu_{t-\tau}(y),
\end{equation*}
By Proposition \ref{prop:mono}, $\Theta(\mathcal{M}, X,\tau)$ is increasing in $\tau$, and thus the limit as $\tau \to 0^+$, referred to as the \textbf{Gaussian density}, exists and is denoted by
\begin{equation*}
\Theta (\mathcal{M},X):=\lim_{\tau\to 0^+}\Theta(\mathcal{M},X,\tau).
\end{equation*}
It is straightforward from the definition that $\Theta(\mathcal M, X)$ is upper semi-continuous in $X$.

\subsection{Level set flow from a mean-convex hypersurface}

Let $\{M_t\}_{t \ge 0}$ be the level set flow starting from a closed, smooth, mean-convex hypersurface $M_0$. This flow, $\{M_t\}_{t \ge 0}$, represents the maximal family of the closed set originating from $M_0$ that satisfies the avoidance principle; see \cite[Section 10]{Ilmanen1994EllipticRA}. It can be shown, as in \cite[Corollary 3.3]{white2000size}, that the level set flow is nonfattening, meaning the level sets have no interior. Consequently, it follows from \cite[Theorem 11.4]{Ilmanen1994EllipticRA} that 
\begin{equation*}
\MM=(\mu_t)_{t \ge 0}:=(\mathscr{H}^n\llcorner M_t)_{t\geq 0}
\end{equation*}
defines an integral Brakke flow. 

Next, we recall the following definition from \cite{sheng2009singularity} and \cite{andrews2012noncollapsing}.

\begin{defn}[$\alpha$-noncollapsed] \label{def:alpha}
A smooth, mean-convex hypersurface $M$ enclosing a domain $\Omega$ is \textbf{$\alpha$-noncollapsed} for a constant $\alpha>0$ if for any $x \in M$, there exist closed balls $\bar B_{\mathrm{int}}\subset \Omega,\bar B_{\mathrm{ext}}\subset \R^{n+1}\setminus \mathrm{int}(\Omega)$, each with radius at least $\alpha/H(x)$, that are tangent to $M$ at $x$ from the interior and exterior, respectively.
\end{defn}

If $H(x)=0$ at some point $x \in M$, then being $\alpha$-noncollapsed implies that $\Omega$ is a half-space. It is noteworthy that any compact, smooth, strictly mean-convex hypersurface is $\alpha$-noncollapsed for some $\alpha>0$. An important result from \cite{sheng2009singularity} and \cite{andrews2012noncollapsing} states that the $\alpha$-noncollapsing condition is preserved under mean curvature flow.

By considering the viscosity mean curvature (cf. \cite[Definition 1.3]{haslhofer2017mean1}), the definition of $\alpha$-noncollapsed can be extended to any mean-convex closed set (cf. \cite[Definition 1.4]{haslhofer2017mean1}). We say a level set flow $\mathcal M=\{M_t\}_{t \ge 0}$ is $\alpha$-noncollapsed if each $\Omega_t$, enclosed by $M_t$, is $\alpha$-noncollapsed. In this framework, Haslhofer and Kleiner \cite[Theorem 1.5]{haslhofer2017mean1} established the following result:

\begin{thm}
A level set flow $\mathcal M=\{M_t\}_{t \ge 0}$ starting from a closed, smooth, mean-convex hypersurface $M_0$ is $\alpha$-noncollapsed if $M_0$ is $\alpha$-noncollapsed.
\end{thm}

In \cite{haslhofer2017mean1}, Haslhofer and Kleiner defined a broader class of flows that includes all $\alpha$-noncollapsed compact level set flows with smooth initial conditions.

\begin{defn}[$\alpha$-Andrews flows]\label{andrewsflow}
The class of \textbf{$\alpha$-Andrews flows} is the smallest class of set flows that contains all $\alpha$-noncollapsed compact level set flows with smooth initial condition, which is closed under operations of restriction to smaller time intervals, parabolic rescaling, and passage to Hausdorff limits.
\end{defn}

For $r>0$, we denote the spacetime dilation by $\mathcal{D}_r:(x, t) \to (rx, r^2 t)$ and define $\mathcal{M}_{X,r}=\mathcal{D}_{1/r}(\mathcal{M}-X)$ for any $\alpha$-Andrews flow $\MM$. For any $X \in \MM$ where $\tf(X)$ is not the initial time, a subsequential limit of $\mathcal M_{X_i, r_i}$ with $X_i \to X$ and $r_i \to 0$ is called a \emph{limit flow} at $X$. In particular, a subsequential limit of $\mathcal M_{X, r_i}$ with $r_i \to 0$ is called a \emph{tangent flow} at $X$. 

It was shown in \cite[Theorem 1.11]{haslhofer2017mean1} that any limit flow is a smooth, convex, ancient solution to the mean curvature flow. Furthermore, each tangent flow is given by either a static plane or shrinking generalized cylinders.

\subsection{Quantitative cylindrical singularities}
We define the moduli space $\mathcal E(n,\alpha, \Lambda)$, consisting of all $n$-dimensional level set flows $\MM=\{M_t\}_{t \ge 0}$ that starts from a closed, smooth, mean-convex, $\alpha$-noncollapsed hypersurface $M_0$ with $\lambda(M_0) \le \Lambda$. Hereafter, all level set flows are assumed to belong to this space unless otherwise specified.

Next, we set $\CC^j$ to be the space consisting of the generalized cylinder
\begin{equation*}
\lc \R^j\times S^{n-j}(\sqrt{2(n-j)|t|}) \rc_{t \le 0}
\end{equation*}
of multiplicity one, together with all its rotations. Then, any tangent flow at $X \in \mathcal S^j \setminus \mathcal S^{j-1}$ belongs to $\CC^j$. The entropy of the flow in $\CC^j$ is denoted by $\Theta_j$. From \cite{Sto94}, we have
\begin{equation}\label{entropycyl}
\sqrt{2}<\Theta_0<\Theta_1<\ldots<\Theta_{n-1}<2.
\end{equation}

Next, we introduce the following definition to describe the local behavior near a given point.

\begin{defn}\label{deficylindrical}
A point $X \in \MM$ is said to be \textbf{$(j,\ep,r)$-cylindrical} if for any $s \in [\ep r, \ep^{-1} r]$, 
\begin{equation*}
(\MM_{X,s})_{\tf=-1} \bigcap B(\vec 0^{n+1},\ep^{-1})
\end{equation*}
is a graph over a fixed cylinder $\Sigma=\R^j\times S^{n-j}(\sqrt{2(n-j)})$ with $C^{[\ep^{-1}]}$-norm at most $\ep$. If the $j$-plane $V$ is the $\R^j$-factor of $\Sigma$, then we say $X$ is $(j,\ep,r)$-cylindrical with respect to $L_Y:=(\ef(Y)+V) \times \{\tf(Y)\}$.
\end{defn}
We have the following quantitative rigidity.
\begin{prop}[Quantitative rigidity]\label{quantitativerigidity}
For any $\ep>0$, there exists a constant $\delta=\delta(n,\alpha, \Lambda,\ep)>0$ such that for $X \in \MM \in \mathcal E(n,\alpha, \Lambda)$ and $r>0$, if $\tf(X)>\delta^{-1} r^2$ and
\begin{equation*}
\Theta(\mathcal{M}, X, \delta^{-1} r^2)-\Theta(\mathcal{M}, X, \delta r^2 )\leq \delta^2,
\end{equation*}
then $X$ is $(j,\ep,r)$-cylindrical for some $j\in\{0,1,\ldots,n\}$.
\end{prop}
\begin{proof}
Without loss of generality, we assume $r=1$. Suppose that the conclusion fails. Then we can find $\MM^k \in \mathcal E(n,\alpha, \Lambda)$, $\delta_k\leq k^{-1}$ and $X_k\in\MM^k$ such that 
\begin{equation}\label{gaussianpinch1}
\Theta(\MM^k,X_k,\delta_k^{-1})-\Theta(\MM^k,X_k, \delta_k)\leq\delta^2_k,
\end{equation}
but $\MM^k$ is not $(j,\ep,1)$-cylindrical for any $j\in\{0,1,\ldots,n\}$.

By taking a subsequence, we assume $\MM^k_{X_k,1}\to \MM^\infty$ in the Hausdorff sense. The limit $\MM^\infty$ is an ancient $\alpha$-Andrews flow as defined in Definition \ref{andrewsflow}. By \cite[Theorem 1.11]{haslhofer2017mean1}, $\MM^\infty$ is smooth before the extinction time $T_0$. Note that $\vec 0^{n+1,1}\in \MM^\infty$, we have $T_0\geq0$. By the local regularity theorem of \cite{white2005local} (see also \cite[Appendix C]{haslhofer2017mean1}), we conclude that $\MM^k_{X_k,1}$ converges to $\MM^\infty$ smoothly on $(-\infty,T_0)$. 

Since $\lambda(\MM^k)\leq \Lambda$, it follows from the convergence of Gaussian densities and \eqref{gaussianpinch1} that 
$$\Theta(\MM^{\infty}, \vec{0}^{n+1,1}, \tau)$$
is constant for any $\tau>0$. Therefore, by Proposition \ref{prop:mono}, we conclude that $\MM^{\infty}$ is a self-shrinker. In particular, $T_0=0$, and by \cite[Theorem 1.11]{haslhofer2017mean1}, $\MM^\infty$ must belong to $\CC^j$ for some $j\in\{0,1,\ldots,n\}$. This yields a contradiction.
\end{proof}

For the quantitative singular stratum $\MS^j_{\ep, r}$ (see Definition \ref{def:sin1}), we immediately get that if $j\leq j',\ep\geq\ep',r\leq r'$, then
\begin{equation*}
\MS^{j}_{\ep,r}\subset \MS^{j'}_{\ep',r'}.
\end{equation*}
Moreover, we have the following characterization of $\MS^j$:
\begin{prop} \label{prop:decom1}
With the above definitions, we have
\begin{equation}\label{quansingandsing}
\MS^j=\bigcup_{\ep>0} \mathcal{S}^{j}_{\epsilon}=\bigcup_{\ep>0}\bigcap_{r\in (0,1)}\MS^{j}_{\ep,r}.
\end{equation}
\end{prop}
\begin{proof}
If $X\notin \MS^j$, then the tangent flow at $X$ is given by shrinking cylinders in $\CC^{j'}$ for $j'\geq j+1$. Thus, for any $\ep>0$, there exists a small $r>0$ such that $X$ is $(j',\ep,r)$-cylindrical and hence $X\notin \mathcal{S}^{j}_{\epsilon}$. In other words, $X \notin \bigcup_{\ep>0} \mathcal{S}^{j}_{\epsilon}$. 
Conversely, if $X \notin \bigcup_{\ep>0} \mathcal{S}^{j}_{\epsilon}$, for any $k\in\N$, we can find $r_k >0$, such that $X$ is $(j',k^{-1},r_k)$-cylindrical for some $j' > j$. Consequently, $X \notin \MS^j$.
\end{proof}
For the top stratum, we have the following characterization.

\begin{prop}\label{prop:top}
There exists a constant $\ep=\ep(n, \alpha, \Lambda)>0$ such that
\begin{equation*}
\mathcal S^{n-1}_{\ep}=\mathcal S.
\end{equation*}
\end{prop}
\begin{proof}
Suppose, for contradiction, that there exists a sequence $\MM^k \in \mathcal E(n,\alpha, \Lambda)$, with $\ep_k \to 0$ and $X_k \in \MM^k$ such that $X_k \in \mathcal S^{n-1}(\MM^k) \setminus \mathcal S^{n-1}_{\ep_k}(\MM^k)$. By Definition \ref{def:sin1}, there exists a sequence $r_k>0$ such that $X_k$ is $(n, \ep_k, r_k)$-cylindrical. Consequently, $\MM_{X_k, r_k}$ converges, after taking a subsequence if necessary, to the static plane. By the local regularity theorem of \cite{white2005local}, it follows that $X_k$ is a regular point if $k$ is sufficiently large, leading to a contradiction.
\end{proof}

For later applications, we need the following notion of an independent set, which ensures that a set does not concentrate on a lower-dimensional plane.

\begin{defn}[Independent set]\label{independentpoints}
For $x\in \espace$, we say a subset $S\subset B(x,r)\subset\espace$ is \textbf{$(j,\beta,r)$-independent} in $B(x,r)\subset\espace$ if for any affine $(j-1)$-plane $V$, there exists some $y\in S$ such that
$$|y-V|\geq \beta r.$$
For $X\in\spacetime$, we say $S\subset\pball(X,r)\subset\spacetime$ is $(j,\beta,r)$-independent in $\pball(X,r)$ if the projection of $S$ onto $\espace\times\{ \tf(X)\}$ is $(j,\beta,r)$-independent in $B(\ef(X),r)$. 
\end{defn}

\section{Quantitative uniqueness of cylindrical singularities}

In this section, we prove a quantitative uniqueness theorem for tangent flows at a point $X \in \MM$. The main ideas and techniques are from \cite{colding2015uniqueness} and \cite{colding2018wondering}, see also \cite[Theorem 0.5]{colding2025quantitative} for a similar result.

We begin by considering an open subset $\Omega \subset \mathbb{R}^j\times S^{n-j}(\sqrt{2(n-j)})$ and a family of graphs $\Omega_{u(t)}$ over $\Omega$ for $t \in I$. Here, 
\begin{equation*}
\Omega_{u(t)}:=\{x+u(x,t)\mathbf{n}(x) \mid x \in \Omega\}.
\end{equation*}

We assume that $\Omega_{u(t)}$ evolves by the rescaled MCF for $t \in I$, which is given by the equation:
\begin{equation*}
(\partial_t x)^{\perp}=\mathbf{H}+\frac{x^{\perp}}{2}.
\end{equation*}

Next, we state the following lemma from \cite[Lemma A.48]{colding2015uniqueness}.

\begin{lem}\label{graphicalcontrol}
With above assumptions, there exist positive constants $C$ and $\ep$, both depending only on $n$, such that if $\rVert u(\cdot,t)\rVert_{C^1}<\epsilon, \forall t\in [t_1,t_2] \subset I$, then
\begin{equation*}
\int_{\Omega}|u(x,t_1)-u(x,t_2)|e^{-\frac{|x|^2}{4}} \, dx\leq C\int_{t_1}^{t_2}\int_{\Omega_{u(t)}} \left| \mathbf{H}+\frac{x^{\perp}}{2} \right |e^{-\frac{|x|^2}{4}} \,dxdt.
\end{equation*}
\end{lem}

Now, we prove the main result of this section. 

\begin{thm}[Quantitative uniqueness] \label{quantitativeuniq}
For any constant $\ep>0$, suppose that $\delta \le \delta(n,\alpha, \Lambda, \ep)$, $\tau_1\leq \delta^2 \tau_2$ and
\begin{equation} \label{eq:assu1}
\Theta(X, \tau_2)-\Theta(X,\tau_1)\leq \delta^2,
\end{equation}
for a point $X \in \MM \in \mathcal E(n,\alpha, \Lambda)$. Then, there exists a $j\in\{0,1,\ldots,n\}$ and a cylinder $\Sigma=\R^j\times S^{n-j}(\sqrt{2(n-j)})$ such that for any $r \in [\ep (\delta^{-1} \tau_1)^{\frac 1 2},\ep^{-1} (\delta \tau_2)^{\frac 1 2}]$, 
\begin{equation*}
(\MM_{X,r})_{\tf=-1} \bigcap B(\vec 0^{n+1},\ep^{-1})
\end{equation*}
is a graph over $\Sigma$ with $C^{[\ep^{-1}]}$-norm at most $\ep$.

In particular, there exists a fixed $j$-plane $V \subset \R^{n+1}$ such that $X$ is $(j, \ep, r)$-cylindrical with respect to $L_X:=(\ef(X)+V) \times \{\tf(X)\}$ for any $r \in [(\delta^{-1} \tau_1)^{\frac 1 2},(\delta \tau_2)^{\frac 1 2}]$.
\end{thm}
\begin{proof}
In the proof, we define $B_i:=B(\vec{0}^{n+1}, i \ep^{-1})$ for $i \in \mathbb N$. We set $X=(x_0,t_0)$ and consider the corresponding rescaled MCF:
\begin{equation*}
M'_{s}:= e^{\frac{s}{2}}(M_{t_0-e^{-s}}-x_0), \quad s=-\log (t_0-t).
\end{equation*}
In addition, we set $s_2=-\log (\delta^{-1} \tau_1)$ and $s_1=-\log (\delta \tau_2)$. By our assumption \eqref{eq:assu1} and Proposition \ref{quantitativerigidity}, if $\delta$ is sufficiently small, then for any $s \in [s_1, s_2]$, there exists a cylinder $\Sigma_s$ of the form $\mathbb{R}^j\times S^{n-j}(\sqrt{2(n-j)})$ such that $M'_{z} \bigcap B_4,$ for $z\in [s+4\log \ep,s-4\log \ep]$ can be written as a graph over $\Sigma_{s}$ with $C^{4[\ep^{-1}]}$-norm smaller than $\ep_0$, where $\ep_0$ is a constant much smaller than $\ep$, to be determined later. Here, since the entropies of the generalized cylinders are discrete (see \eqref{entropycyl}), we may assume that $j$ is fixed for all $s \in [s_1,s_2]$.

We assume $s' \in [s_1, s_2]$ is the largest number such that $F(M'_{s}) \ge \Theta_j$ for any $s \in [s_1, s']$. Without loss of generality, we may assume $s'=s_2$, as the general case can be handled by considering the intervals $[s_1,s']$ and $[s', s_2]$ separately.

It follows from \cite{colding2015uniqueness} that the following discrete {\L}ojasiewicz inequality holds: there exists $\eta\in (1/3,1)$ such that for any $s\in [s_1,s_2]$, 
\begin{equation}\label{discretelojasiewicz}
(F(M'_{s})-\Theta_j)^{1+\eta}\leq K \lc F(M'_{s-1})-F(M'_{s+1}) \rc,
\end{equation}
where the constant $K$ depends only on $n$ and $\Lambda$. Next, we set $G(s):=K^{-\frac 1 \eta} (F(M'_{s+s_1})-\Theta_j)$ for $s \in [0,s_2-s_1]$ and define $C_0>0$ by $C_0^{\eta}=8\eta^{-1}$.
\begin{claim}\label{keyclaim}
For any $s \in [0, s_2-s_1]$, we have
\begin{equation}\label{decayofG}
G(s)\leq C_0[G(0)^{-\eta}+s]^{-\frac{1}{\eta}}.
\end{equation}
\end{claim}
\emph{Proof of Claim \ref{keyclaim}}: We first show that \eqref{decayofG} holds for any $s \in [0, \min\{2, s_2-s_1\}]$. Since $G(s)$ is decreasing, this follows from
\begin{equation}\label{decayofGa}
G(0)\leq C_0[G(0)^{-\eta}+2]^{-\frac{1}{\eta}}.
\end{equation}
By our assumption \eqref{eq:assu1} and \eqref{discretelojasiewicz}, $G(0)=\Psi(\delta)$, which ensures that \eqref{decayofGa} holds.

Next, we show that if the inequality \eqref{decayofG} holds for $s-2$, then it also holds for $s$. Suppose not; then we would have
\begin{align}\label{contradictineq}
G(s-2)\leq C_0[G(0)^{-\eta}+s-2]^{-\frac{1}{\eta}} \quad \text{and} \quad G(s) \geq C_0[G(0)^{-\eta}+s]^{-\frac{1}{\eta}}.
\end{align}
By \eqref{discretelojasiewicz}, we have
$$G(s)^{1+\eta}\leq G(s-1)-G(s+1)$$
and
$$G(s)^{1+\eta}\leq G(s-1)^{1+\eta}\leq G(s-2)-G(s).$$
Therefore,
$$G(s)(1+G(s)^{\eta})\leq G(s-2).$$
By \eqref{contradictineq}, we get
\begin{align*}
[G(0)^{-\eta}+s]^{-\frac{1}{\eta}}(1+C_0^{\eta}[G(0)^{-\eta}+s]^{-1})\leq [G(0)^{-\eta}+s-2]^{-\frac{1}{\eta}}.
\end{align*}
After simplification, it is equivalent to
\begin{align} \label{contradictineq2}
(1+C_0^{\eta}x)^{-\eta}\geq 1-2x,
\end{align} 
where $x=[G(0)^{-\eta}+s]^{-1}$. Since $x$ is sufficiently small, we have $(1+C_0^\eta x)^{-\eta}\leq 1-\frac{1}{2}\eta C_0^{\eta}x=1-4x$ by our definition of $C_0$. This contradicts \eqref{contradictineq2}, and the proof of the claim is complete.

It follows from Proposition \ref{prop:mono} that for any $s \in [0, s_2-s_1]$,
\begin{equation*}
\int_{s}^{s_2-s_1}\int_{M'_{s_1+z}}\left| \mathbf H+\frac{x^{\perp}}{2} \right|^2 e^{-\frac{|x|^2}{4}}\,dxdz\leq C[G(0)^{-\eta}+s]^{-\frac{1}{\eta}}.
\end{equation*}
Therefore, one can apply the same argument as in \cite[Lemma 2.39]{colding2018wondering} to conclude that for $\gamma\in (1,\eta^{-1})$, there exists a constant $C=C(n,\Lambda,\gamma,\eta)$ such that
\begin{equation}\label{weightedintegralestimate}
\int_{1}^{s_2-s_1} z^{\gamma}\int_{M'_{s_1+z}}\left| \mathbf H+\frac{x^{\perp}}{2} \right|^2e^{-\frac{|x|^2}{4}}\,dxdz\leq C G(0)^{1-\gamma\eta}.
\end{equation}
Indeed, we have
\begin{align*}
& \int_{1}^{s_2-s_1} z^{\gamma}\int_{M'_{s_1+z}}\left| \mathbf H+\frac{x^{\perp}}{2} \right|^2e^{-\frac{|x|^2}{4}}\,dxdz \\
\leq & \sum_{j=0}^\infty \int_{2^j}^{\min\{s_2-s_1, 2^{j+1}\}} z^{\gamma}\int_{M'_{s_1+z}} \left| \mathbf H+\frac{x^{\perp}}{2} \right|^2e^{-\frac{|x|^2}{4}}\,dxdz\\
\leq & \sum_{j=0}^\infty \int_{2^j}^{\min\{s_2-s_1, 2^{j+1}\}} 2^{\gamma(j+1)}\int_{M'_{s_1+z}} \left| \mathbf H+\frac{x^{\perp}}{2} \right|^2e^{-\frac{|x|^2}{4}}\,dxdz\\
\leq& C\sum_{j=0}^\infty 2^{\gamma (j+1)}[G(0)^{-\eta}+2^j]^{-\frac{1}{\eta}} \le C [1/2+G(0)^{-\eta}]^{\gamma-\eta^{-1}} \leq CG(0)^{1-\eta\gamma}.
\end{align*}
Here, the forth inequality holds since (see \cite[Lemma 2.37]{colding2018wondering}):
\begin{align*}
\sum_{j=0}^\infty 2^{\gamma (j+1)}[G(0)^{-\eta}+2^j]^{-\frac{1}{\eta}}& \leq 4^\gamma\sum_{j=0}^\infty \min_{r\in [2^{j-1},2^j]}r^{\gamma}[G(0)^{-\eta}+r]^{-\frac{1}{\eta}}\\
& \le 4^\gamma \sum_{j=0}^\infty 2^{1-j}\int_{2^{j-1}}^{2^j}r^{\gamma}[G(0)^{-\eta}+r]^{-\frac{1}{\eta}}dr\\
& \le 2 \cdot 4^\gamma \int_{1/2}^\infty [G(0)^{-\eta}+r]^{\gamma-1-\frac{1}{\eta}}dr \leq C [G(0)^{-\eta}+1/2]^{\gamma-\frac{1}{\eta}}.
\end{align*}

Now, we assume $\bar s$ is the largest number in $[s_1,s_2]$ such that for any $s \in [s_1+2\log \ep, \bar s-2\log \ep]$, $M'_s \bigcap B_1$ can be written as a graph $u$ over $\Sigma_{s_1}$ with $C^{[\ep^{-1}]}$-norm at most $\ep$.

If $\bar s < s_2$, we apply Lemma \ref{graphicalcontrol} to get that for any $[t_1,t_2]\subset [s_1,\bar s]$,
\begin{align}\label{L1closenessofgraph}
& \lc \int_{\Sigma_{s_1} \cap B_1}|u(x,t_1)-u(x,t_2)|e^{-\frac{|x|^2}{4}}\,dx \rc^2 \notag \\
\leq & C \lc \int_{t_1}^{t_2}\int_{M'_s}\left| \mathbf H+\frac{x^{\perp}}{2} \right|e^{-\frac{|x|^2}{4}}\,dxds \rc^2\nonumber\\
\leq & C \lc \int_{t_1}^{t_2} (s-s_1)^{-\gamma} \int_{M'_s} e^{-\frac{|x|^2}{4}}\,dx ds \rc \lc \int_{t_1}^{t_2} \int_{M'_s} (s-s_1)^{\gamma} \left| \mathbf H+\frac{x^{\perp}}{2} \right|^2 e^{-\frac{|x|^2}{4}}\,dxds \rc \nonumber\\ 
\leq & C F(M'_{s_1}) \int_{t_1}^{t_2} \int_{M'_s} (s-s_1)^{\gamma} \left| \mathbf H+\frac{x^{\perp}}{2} \right|^2 e^{-\frac{|x|^2}{4}}\,dxds \le CG(0)^{1-\eta\gamma},
\end{align}
where we have used \eqref{weightedintegralestimate} for the last inequality.

Since $M'_{s_1} \bigcap B_4$ is a graph over $\Sigma_{s_1}$ with $C^{4[\ep^{-1}]}$-norm bounded by $\ep_0$, it follows from \eqref{L1closenessofgraph} and interpolation (cf. \cite[Lemma B.1]{colding2015uniqueness}) that if $\delta$ is sufficiently small, $M'_{\bar s} \bigcap B_{0.5}$ is also a graph over $\Sigma_{s_1}$ with $C^{3[\ep^{-1}]}$-norm bounded by $2\ep_0$. Additionally, $M'_{\bar s} \bigcap B_4$ is a graph over $\Sigma_{\bar s}$ with $C^{4[\ep^{-1}]}$-norm bounded by $\ep_0$. Consequently, the Hausdorff distance between $\Sigma_{s_1}$ and $\Sigma_{\bar s}$ in $B_2$ is at most $C_n \ep_0$. Thus, for any $s \in [\bar s+4 \log \ep, \bar s-4\log \ep]$, $M'_s \bigcap B_1$ is a graph over $\Sigma_{s_1}$ with $C^{[\ep^{-1}]}$-norm bounded by $C_n \ep_0$. Since $\ep_0$ is much smaller than $\ep$, this contradicts the definition of $\bar s$, leading us to conclude that $\bar s=s_2$.

In conclusion, we have shown that $M'_s \bigcap B_1$ for $s\in [s_1+2 \log \ep,s_2-2\log \ep]$ can be written as a graph over $\Sigma_{s_1}$ with $C^{[\ep^{-1}]}$-norm bounded by $\ep$. Thus, the proof of the theorem is complete.

\end{proof}

\section{Neck decomposition theorem}
In this section, we consider a level set flow $\MM=\{M_t\}_{t \ge 0} \in \mathcal E(n,\alpha, \Lambda)$. To proceed, we first introduce the following concept of a neck region in our setting, which is similar to those in \cite{jiang20212} and \cite{Cheeger2018RectifiabilityOS}.

\begin{defn}[Neck region]\label{defiofneckregion}
Fix $j\in \{0,1,\ldots,n-1\}, \delta>0, r>0$ and a universal constant $c_n=100^{-10n}$. Given $X\in \MM$ with $\tf(X)>2\delta^{-1} r^2$, we call a subset $\nnn\subset \pball(X,2r)$ a \textbf{$(j,\delta, r)$-neck region} if $\nnn=\pball(X,2r)\setminus \pball_{r_Y}(\ccc)$, where $\ccc\subset \pball(X,2r) \bigcap \MM$ is a closed subset with $r_Y:\ccc\to\mathbb{R}_{\geq 0}$, satisfies:
\begin{itemize}[leftmargin=*, label={}]
\item \emph{(n1)} $\{\pball(Y,c_n^3 r_Y)\}_{Y\in\ccc}$ are pairwisely disjoint;
\item \emph{(n2)} for all $Y\in\ccc$,
$$\Theta(Y, \delta^{-1} r^2)-\Theta(Y, \delta r_Y^2)<\delta^2;$$
\item \emph{(n3)} for any $Y \in \ccc$, there exists a $j$-plane $V_Y$ such that for all $r_Y\leq s\leq c_n^{-3} r$, $Y$ is $(j,\delta, s)$-cylindrical with respect to $L_Y= (V_Y+\ef(Y)) \times \{\tf(Y)\}$ (cf. Definition \ref{deficylindrical});
\item \emph{(n4)} for all $Y\in\ccc$ and $s \ge r_Y$ with $\pball(Y,2s)\subset \pball(X, 2r)$, we have 
\begin{equation*}
L_Y \bigcap \pball(Y,s)\subset \pball_{c_ns}(\ccc) \quad \text{and} \quad \ccc\bigcap \pball(Y,s)\subset \pball_{c_ns}(L_Y).
\end{equation*}
\end{itemize}
Here, $\ccc$ is called the \textbf{center of the neck region}, and $r_Y$ is referred to as the \textbf{radius function}. We decompose $\ccc=\ccc_{+}\bigcup \ccc_0$, where $r_Y>0$ on $\ccc_+$ and $r_Y=0$ on $\ccc_0$. In addition, we use the notation
\begin{equation*}
\pball_{r_Y}(\ccc):=\ccc_0 \bigcup \bigcup_{Y \in \ccc_+} \pball(Y, r_Y).
\end{equation*}
\end{defn}

The following lemma follows directly from the Definition \ref{defiofneckregion} (n1).
\begin{lem}
The radius function $r_Y$ is a Lipschitz function with Lipschitz constant at most $c_n^{-3}$.
\end{lem}

Next, we show that all points in the center have nearly identical time components, and their corresponding affine planes are close.

\begin{lem} \label{lem:x01}
For any $\ep>0$, let $\nnn=\pball(X,2r)\setminus \pball_{r_Y}(\ccc)$ be a $(j,\delta,r)$-neck region with $\delta \le \delta(n,\ep)$. Then, for any $Y_1,Y_2 \in \ccc$ with $s=\|Y_1-Y_2\|$,
\begin{equation} \label{eq:times1}
|\tf(Y_1)-\tf(Y_2)|\le \ep^2 s^2
\end{equation}
and
\begin{equation} \label{eq:times2}
d_H \lc (V_{Y_1}+\ef(Y_1) )\bigcap B(\ef(Y_1),10 s),(V_{Y_2}+\ef(Y_2))\bigcap B(\ef(Y_1),10 s)\rc\leq \ep s,
\end{equation} 
where $d_H$ denotes the Hausdorff distance in $\R^{n+1}$. 
\end{lem}

\begin{proof}
Since $\pball(Y_1,c_n^3 r_{Y_1})$ and $\pball(Y_2,c_n^3 r_{Y_2})$ are disjoint, we have $s \ge c_n^3 \max\{r_{Y_1}, r_{Y_2}\}$. By the definition of the neck region, both $Y_1$ and $Y_2$ are $(j,\delta, c_n^{-3} s)$-cylindrical with respect to $L_{Y_1}$ and $L_{Y_2}$, respectively. Consequently, it follows from Definition \ref{deficylindrical} that if $\delta$ is sufficiently small, the parabolic Hausdorff distance between $L_{Y_1}$ and $L_{Y_2}$ in $\pball(Y_1, 10\ep)$ is smaller than $\ep s$. In particular, $|\tf(Y_1)-\tf(Y_2)| \le\ep^2 s^2$ and that the Hausdorff distance between $(V_{Y_1}+\ef(Y_1) )$ and $(V_{Y_2}+\ef(Y_2) )$ in the ball $B(\ef(Y_1),10 s)$ is less than $\ep s$.
\end{proof}

Using Lemma \ref{lem:x01}, we derive the following structural description of a neck region.
\begin{thm}[Neck structure]\label{neckstructurethm}
For any $\ep>0$, let $\nnn=\pball(X,2r)\setminus \pball_{r_Y}(\ccc)$ be a $(j,\delta,r)$-neck region with $\delta \le \delta(n,\ep)$. Then, the following conclusion holds.

For a fixed $Y_0 \in \ccc$, let $V=V_{Y_0}$ and let $\pi: \ccc \to V$ be the projection onto $V$. Then, for any $Y_1,Y_2 \in \ccc$,
\begin{equation}\label{bilipschitz}
(1-\ep) \rVert Y_1-Y_2\rVert \leq |\pi(Y_1)-\pi(Y_2)|\leq \rVert Y_1-Y_2\rVert.
\end{equation}
In particular, $\ccc$ is a Lipschitz graph over $V$.
\end{thm}

\begin{proof}
The second inequality in \eqref{bilipschitz} is straightforward, so we only need to prove the first. Without loss of generality, we assume $r=1$.

For any $Y_1=(y_1,t_1), Y_2=(y_2,t_2) \in \ccc$, set $s:=\|Y_1-Y_2\|$, $V_1:=V_{Y_1}$ and $V_2:=V_{Y_2}$. First, by \eqref{eq:times1}, we have $|t_1-t_2| \le \ep^2 s^2$, which implies $|y_1-y_2|=s$, if $\delta$ is sufficiently small. Additionally, from \eqref{eq:times2}, we obtain
\begin{equation*}
d_H \lc (y_1+V_1)\bigcap B(y_1,10 s),(y_2+V_2))\bigcap B(y_1,10 s)\rc\leq \ep s,
\end{equation*} 
and consequently,
\begin{align} \label{eq:times3}
|\pi_1(Y_1)-\pi_1(Y_2)| \ge \sqrt{1-\ep^2} s,
\end{align}
where $\pi_1$ is the projection onto $V_1$. Moreover, by \eqref{eq:times2} again, we have
\begin{align}\label{eq:times4}
d_H \lc V_1 \bigcap B(\vec{0}^{n+1},1), V \bigcap B(\vec{0}^{n+1},1)\rc\leq \ep.
\end{align}
Combining \eqref{eq:times3} with \eqref{eq:times4}, we immediately conclude that
\begin{align*}
|\pi(Y_1)-\pi(Y_2)| \ge (1-2\ep) s.
\end{align*}
\end{proof}

For later applications, we define the \textbf{packing measure} associated with the neck region $\nnn=\pball(X,2 r)\setminus \pball_{r_Y}(\ccc)$ by
\begin{equation} \label{eq:pack}
\mu :=\sum_{Y\in\ccc_+}r_Y^j\delta_{Y}+\mathscr{H}_P^j \llcorner \ccc_0.
\end{equation}

Next, we establish the following Ahlfors regularity property for the packing measure.

\begin{prop} \label{prop:ahlfors}
For any $Y\in \ccc,s\geq r_Y$ with $\pball(Y,2s)\subset \pball(X,2r)$, we have
\begin{equation}\label{Ahlforsregularity}
C(n)^{-1}s^j\leq\mu(\pball(Y,s))\leq C(n) s^j.
\end{equation}
\end{prop}

\begin{proof}
Let $V=V_Y$ and denote the projection onto $V$ by $\pi$.

We first prove
\begin{equation}\label{ahlforsregularity}
\sum_{Y\in\ccc_{+}}r_Y^j+\mathscr{H}_P^j(\ccc_0)\leq C(n) r^j.
\end{equation}

Since, by our definition, the sets $\{\pball (Y, c_n^3r_Y)\}$ are disjoint, it follows from \eqref{bilipschitz} that $\{B(\pi(Y),c_n^3r_Y/2)\}$ are disjoint in $V$. Consequently, we have
\begin{equation}\label{eq:times5}
\sum_{Y\in\ccc_+}r_Y^j\leq C(n)\sum_{Y\in\ccc_+}|B(\pi(Y),c_n^3r_Y/2)|_V\leq C(n)r^j,
\end{equation} 
where $|\cdot|_V$ denotes the volume on the plane $V$.

Additionally, it follows from \eqref{bilipschitz} that
\begin{equation}\label{eq:times6}
\mathscr H_P^j(\ccc_0)\leq C(n)\mathscr H^j(\pi(\ccc_0))\leq C(n) r^j.
\end{equation} 

Combining \eqref{eq:times5} and \eqref{eq:times6}, we immediately conclude that \eqref{ahlforsregularity} holds.

An important observation here is that $\nnn'=\nnn \cap \pball(Y, 2s)$ is also a $(j,\delta, s)$-neck region with center $\ccc'=\ccc \cap \pball(Y, 2s)$ and the same radius function. Therefore, by \eqref{ahlforsregularity}, the upper bound in \eqref{Ahlforsregularity} holds. 

To establish the lower bound of \eqref{ahlforsregularity}, we consider the containment:
\begin{equation}\label{eq:times7}
V\bigcap B(\pi(Y),s/2) \subset \bigcup_{Z\in \ccc \bigcap \pball(Y,s)} V\bigcap \bar{B}(\pi(Z), r_Z),
\end{equation} 
where $\bar{B}(\pi(Z), r_Z)=\pi(Z)$ if $Z \in \ccc_0$. By \eqref{bilipschitz} and \eqref{eq:times7}, we have
\begin{align*}
C(n)^{-1}s^j\leq |V\bigcap \bar{B}(\pi(Y),s/2)| \leq \sum_{Z\in \ccc_+ \bigcap \pball(Y,s)}C(n)r_Z^j+C(n) \mathscr{H}_P^j \lc\ccc_0 \bigcap \pball(Y,s) \rc \le C(n) \mu(\pball(Y,s)).
\end{align*}
If \eqref{eq:times7} were to fail, there would exist $w \in V\bigcap B(\pi(Y),s/2)$ not covered by the union. Define $s_Z=|w-\pi(Z)|$ for any $Z\in \ccc \bigcap \pball(Y,s)$ and set $\bar s=\inf_{Z\in \ccc \bigcap \pball(Y,s)} s_Z>0$. Assume $\bar s=s_{W}>r_{W}$ for some $W\in \ccc \bigcap \pball(Y,s)$. Then, by Lemma \ref{lem:x01}, there exists $W' \in L_{W}$ such that $|\pi(W')-w| \le c_n \bar s$. Further, by Definition \ref{defiofneckregion} (n4),
\begin{align*}
L_{W}\bigcap \pball(W, 2\bar s)\subset \bigcup_{Z\in \ccc\bigcap \pball(W, 2\bar s)} \pball(Z, 2c_n \bar s).
\end{align*}
Thus, there exists $W'' \in \ccc\bigcap \pball(W, 2\bar s)$ such that $\|W''-W'\| \le 2c_n \bar s$. This implies
\begin{align*}
|w-\pi(W'')| \le |w-\pi(W')|+\|W'-W''\| \le 3c_n \bar s<\bar s,
\end{align*}
which contradicts the definition of $W$, thereby confirming that \eqref{eq:times7} holds.

In conclusion, the proof is complete.
\end{proof}

For later applications, we show that any point in the neck region is sufficiently regular.

\begin{lem} \label{lem:regu}
For any $\ep>0$, there exists $\chi=\chi(n,\ep)>0$ such that if $\nnn=\pball(X,2 r)\setminus \pball_{r_Y}(\ccc)$ is a $(j,\delta,r)$-neck region with $\delta \le \delta(n,\ep)$, then any $Y \in \nnn \bigcap \MM \bigcap \pball(X,r)$ is $(n, \ep, s)$-cylindrical, provided that $s \le \chi \|Y - \ccc\|$.
\end{lem}

\begin{proof}
Set $d:=\|Y - \ccc\|=\rVert Y-Y_1\rVert$, where $Y_1\in \ccc$. Then $d \geq r_{Y_1}$, and $Y_1$ is $(j,\delta,d)$-cylindrical. If $\delta$ is suffciently small, there exists $\ep_n>0$ depending only on $n$ such that $|\tf(Y)-\tf(Y_1)| \ge \ep^2_n d^2$. Indeed, if $|\tf(Y)-\tf(Y_1)| \le \ep^2_n d^2$, then $\|Y- L_{Y_1}\|\le d/100$. However, by Definition \ref{defiofneckregion} (n4), there exists $Y_1'\in\ccc$ such that $\rVert Y-Y_1'\rVert\leq d/100$, contradicting the assumption that $d=\|Y - \ccc\|$.

Furthermore, by the clearing-out lemma (see \cite[Lemma 6.3]{brakke2015motion} \cite[Proposition E.5]{ecker2012regularity} or \cite[Theorem 6.1]{colding2016singular}), we have $\tf(Y) \le \tf(Y_1)-\ep^2_n d^2$, provided $\delta$ is small. This implies that $Y$ is $(n,\ep, \chi d)$-cylindrical, if $\delta$ and $\chi$ are sufficiently small.
\end{proof}

Next, we state the main result of this section.

\begin{thm}[Neck decomposition theorem]\label{neckdecomposition}
For any $\eta>0$, $r>0$ and $j\in\{0,1,\ldots,n-1\}$, suppose that $\delta\leq\delta(n,\alpha, \Lambda,\eta)$ and $X \in \MM \in \mathcal E(n,\alpha, \Lambda)$ with $\tf(X)>4\zeta^{-2}r^2$, for some $\zeta=\zeta(n,\alpha, \Lambda, \delta, \eta)$. Then, we have the following decomposition
\begin{equation*}
\pball(X, r) \bigcap \MM \subset \bigcup_a \lc (\nnn_a\bigcup \ccc_{0,a} )\bigcap \pball(Y_a,r_a) \rc \bigcup_b \pball(Y_b,r_b)\bigcup S_0
\end{equation*}
with the following properties:
\begin{enumerate}[label=(\alph*)]
\item For all $a$, $\nnn_a=\pball(Y_a, 2r_a)\setminus \pball_{r_{a,Y}}(\ccc_{a})$ is a $(j,\delta, r_a)$-neck region.
\item For all $b$, \bballY.
\item The following content estimate holds:
\begin{equation*}
\sum_a r_a^j+\sum_b r_b^j+\mathscr{H}_P^j(\bigcup_a\ccc_{0,a})\leq C(n,\alpha,\Lambda,\delta,\eta) r^j.
\end{equation*}
\item $S_0 \subset \MS^{j-1}$. In particular, $\mathscr{H}_P^j(S_0)=0$.
\item If $\eta\leq\eta(n,\epsilon)$ and $\delta\leq\delta(n,\alpha,\Lambda,\ep)$, then
$$\pball(X, r) \bigcap\mathcal{S}^j_{\epsilon}\subset S_0\bigcup\bigcup_a \ccc_{0,a}.$$
\end{enumerate}
\end{thm}

Assuming Theorem \ref{neckdecomposition}, we derive the following structural description of the singular stratum:

\begin{cor} \label{cor:quali}
Under the same assumptions as in Theorem \ref{neckdecomposition}, 
\begin{equation*}
\pball(X, r) \bigcap(\mathcal{S}^j \setminus \MS^{j-1})
\end{equation*}
is contained in countably many compact $j$-dimensional Lipschitz graphs.
\end{cor}

\begin{proof}
Let $\ep_k \to 0$. By Theorem \ref{neckdecomposition} parts (d) and (e), we have
\begin{equation*}
\pball(X, r) \bigcap(\mathcal{S}_{\ep_k}^j \setminus \MS^{j-1}) \subset \bigcup_{a_k} \ccc_{0,a_k},
\end{equation*}
where each $\ccc_{0,a_k}$ is part of the center of a $(j, \delta_k, r_{a_k})$-neck region. Since $\MS^j=\bigcup_{k} \MS^j_{\ep_k}$, it follows that
\begin{equation*}
\pball(X, r) \bigcap(\mathcal{S}^j \setminus \MS^{j-1}) \subset \bigcup_k \bigcup_{a_k} \ccc_{0,a_k}.
\end{equation*}
From Theorem \ref{neckstructurethm}, each $\ccc_{0,a_k}$ is a compact $j$-dimensional Lipschitz graph. This completes the proof.
\end{proof}

Note that for any $\MM \in \mathcal E(n,\alpha, \Lambda)$, we have $\MM \subset \pball(\vec{0}^{n+1,1}, r)$ for a sufficiently large $r$ depending on the initial hypersurface. Consequently, Theorem \ref{thm:quali} follows directly from Corollary \ref{cor:quali} by a covering argument. The only additional point is that $\MS^{n-1} \setminus \MS^{n-2}$ is compact.

Before giving the proof of Theorem \ref{neckdecomposition}, let us provide an example to illustrate the theorem.

\begin{exmp}
We consider the standard shrinking cylinders $\MM=\{\R^k\times S^{n-k}(\sqrt{2(n-k)|t|})\}_{t\leq 0}$ with $X=\vec{0}^{n+1,1}$. The neck decomposition of $\pball (X,1)$ can be obtained as follows with three different cases. 
\begin{enumerate}
\item \textbf{Case} $0\leq j\leq k-1$: The decomposition consists of a single parabolic ball $\pball(Y_b, r_b)=\pball (X,1)$. In particular, $X$ is $(k, \eta, r_b)$-cylindrical for any $\eta>0$.

\item \textbf{Case} $j=k$: We set $\pball(Y_a,r_a)=\pball(X,1), \nnn_a=\pball(Y_a,2 r_a)\setminus \ccc_0$ with $\ccc=\ccc_0=\R^k\times \{\vec 0^{n+1-k}\}\times \{0\}$. In particular, $\nnn_a$ is a $(j, \delta, r_a)$-neck region for any $\delta>0$.

\item \textbf{Case} $j>k$: We set $S_0=\R^k\times \{\vec 0^{n+1-k}\}\times \{0\}$. For a fixed $\eta>0$ and for each $r_b=2^{-b}$, we choose a cover $\{\pball(Y_{r_b,l}, \sigma r_b)\}_{1\leq l \leq N_{r_b}}$ of
\begin{equation*}
\pball(X,1) \bigcap \MM \bigcap \lc \pball_{2r_b}(S_0)\setminus \overline{\pball}_{r_b}(S_0)\rc,
\end{equation*}
such that $Y_{r_b,l} \in \MM$ and $\{\pball(Y_{r_b,l}, \sigma r_b/2)\}$ are pairwise disjoint. Here, $\sigma= \sigma(n,\eta)$ is chosen such that each $Y_{r_b,l}$ is $(n,\eta,\sigma r_b)$-cylindrical. Moreover, $N_{r_b}\leq C(n,\eta)r_b^{-k}$. Therefore, we have the following decomposition:
\begin{align*}
\pball(X,1)\bigcap \MM \subset \lc \bigcup_{0<r_b=2^{-b}\leq1}\bigcup_{l=1}^{N_{r_b}}\pball(Y_{r_b,l}, \sigma r_b) \rc \bigcup S_0.
\end{align*}
Additionally, the content estimate holds since
\begin{align*}
\sum_{0<r_b=2^{-b}\leq 1}\sum_{l=1}^{N_{r_b}}r_b^j\leq C(n, \eta)\sum_{0<r_b=2^{-b}\leq 1}r_b^{-k}r_b^j=C(n, \eta)\sum_{0\leq b<\infty}2^{-b(j-k)} =C(n, \eta).
\end{align*}
\end{enumerate}
\end{exmp}

The remainder of this section is devoted to proving Theorem \ref{neckdecomposition}. Without loss of generality, we assume $r=1$. Throughout, we will use the following notation:
$$\bar V_{Y,r}(\zeta):= \sup_{Z\in \pball(Y,4r) \bigcap \MM}\Theta(Z,\zeta^{-2}r^2),$$
where $\zeta$ is a small parameter.

\begin{defn}[Pinching set]
Given $X \in \MM$ and $\zeta>0$, we set $\bar V :=\bar V_{X,1}(\zeta)$ and for any $Y\in\MM\bigcap\pball(X,2)$,
\begin{equation*}
\VV_{\zeta,r}(Y):=\{Z\in \pball(Y,4r) \bigcap \MM \mid \Theta( Z,(\zeta r)^2)>\bar V-\zeta\}.
\end{equation*}
\end{defn}

\begin{defn}[Different types of balls]\label{allkindsofballs} Fix $j\in \{0,1,\ldots,n-1\}, \zeta>0,\eta>0$ and $\beta>0$, we define the following different types of balls centered at some point in $\MM$.
\begin{enumerate} [label=(\alph*)]
\item A ball $\pball(Y_a,r_a)$ is called an $a$-ball if $\nnn_a=\pball(Y_a, 2 r_a)\setminus \pball_{r_{a,Y}}(\ccc_a)$ is a $(j,\delta, r_a)$-neck region;
\item A ball $\pball(Y_b,r_b)$ is called a $b$-ball if \bballY;
\item A ball $\pball(Y_c,r_c)$ is called a $c$-ball if it is not a $b$-ball and there exists $Z \in \mathcal{V}_{\zeta, r_c}(Y_c)$ such that $\Theta(Z, (\zeta r_c)^2) > \bar V-\zeta/2$, and $\mathcal{V}_{\zeta, r_c}(Y_c)$ is $(j,\beta,4r_c)$-independent (cf. Definition \ref{independentpoints});
\item A ball $\pball(Y_d,r_d)$ is called a $d$-ball if it is not a $b$-ball and there exists $Z \in \mathcal{V}_{\zeta, r_d}(Y_d)$ such that $\Theta(Z, (\zeta r_d)^2) > \bar V-\zeta/2$, and $\mathcal{V}_{\zeta, r_d}(Y_d)$ is not $(j,\beta,4r_d)$-independent;
\item A ball $\pball(Y_e,r_e)$ is called an $e$-ball if it is not a $b$-ball and $\Theta(Z, (\zeta r_e)^2) \le \bar V-\zeta/2$ for any $Z \in \pball(Y_e, 4 r_e) \bigcap \MM$.
\end{enumerate}
\end{defn}
Note that the constant $\eta$ is only used for defining $b$-balls, while $\beta$ is used to define $c$-balls and $d$-balls. Furthermore, every ball must belong to one of the categories: $b$-ball, $c$-ball, $d$-ball, or $e$-ball.

The following two propositions give further decompositions of $c$-balls and $d$-balls.

\begin{prop}[Decomposition of $c$-balls]\label{decompositioncball}
For $Y\in\MM,s>0$ with $\tf(Y)>4\zeta^{-2}s^2$, let $\pball(Y,s)$ be a $c$-ball with $\pball(Y,4s)\subset \pball(X,4)$. If $\delta\leq\delta(n,\alpha,\Lambda), \zeta\leq\zeta(n,\alpha,\Lambda,\delta,\beta,\eta)$ , then we have the decomposition
\begin{align*}
\pball(Y,2s)\subset &(\ccc_0\bigcup \nnn)\bigcup\bigcup_b \pball(Z_b,r_b)
\bigcup\bigcup_d \pball(Z_d,r_d) \bigcup \bigcup_e\pball(Z_e,r_e)
\end{align*}
satisfying
\begin{enumerate} [label=(\roman*)]
\item for each $b$, $\pball(Z_b,r_b)$ is a $b$-ball;
\item for each $d$, $\pball(Z_d,r_d)$ is a $d$-ball;
\item for each $e$, $\pball(Z_e,r_e)$ is an $e$-ball;
\item $\nnn := \pball(Y,2s)\setminus \lc \ccc_0\bigcup \bigcup_b \pball(Z_b,r_b) \bigcup\bigcup_d \pball(Z_d,r_d)\bigcup \bigcup_e\pball(Z_e,r_e) \rc$ is a $(j,\delta, s)$-neck region;
\item the following content estimate holds:
\begin{equation*}
\sum_b r_b^j+\sum_d r_d^j+\sum_e r_e^j+\mathscr{H}_P^j (\ccc_0)\leq C(n)s^j.
\end{equation*}
\end{enumerate}

\end{prop}
\begin{proof}
Without loss of generality, let us assume $s=1$. Since $\pball(Y,1)$ is not a $b$-ball, any $Z \in \mathcal{V}_{\zeta,1}(Y) $ is $(j',\Psi(\zeta),1)$-cylindrical for some $j'\in\{0,1,\ldots,j\}$. We fix $Y_1 \in \mathcal{V}_{\zeta,1}(Y)$. Then, similar to Lemma \ref{lem:x01}, it can be shown that all $Z \in \mathcal{V}_{\zeta,1}(Y) $ are $(j',\Psi(\zeta),1)$-cylindrical for the same $j'$. Moreover, we choose a small constant $\ep>0$, to be determined later, such that if $\zeta$ is sufficiently small, 
\begin{align*}
\mathcal{V}_{\zeta,1}(Y)\subset \pball_{\ep} (L_{Y_1}).
\end{align*}
Since $\mathcal{V}_{\zeta, 1}(Y)$ is $(j,\beta,4)$-independent, we conclude that $j'=j$. In particular, $Y_1$ is $(j,\Psi(\zeta),1)$-cylindrical with respect to $L_{Y_1}$.

We set $\gamma=\frac{1}{100}$ and choose a Vitali cover $\{\pball(Y_{i^1}, c_n^2\gamma)\}$ of $L_{Y_1}\bigcap \pball (Y,4)$ such that $\{\pball(Y_{i^1}, c_n^3\gamma)\}$ are pairwisely disjoint and $Y_{i^1}\in \pball_{\ep}(L_{Y_1}) \bigcap \MM$. Define
$$\nnn^1 := \pball (Y,2)\setminus\bigcup_{i^1}\pball (Y_{i^1},\gamma),$$
then one can verify that if $\zeta\leq\zeta(n,\alpha,\Lambda,\delta,\beta,\eta, \ep)$, $\nnn^1$ is $(j,\delta,1)$-neck region with center $\{Y_{i^1}\}$ and radius function $r_{Y_{i^1}}\equiv \gamma$. In fact, property $(n1)$ follows directly from the construction. For $(n2)$, note that by our choice of $Y_1$,
$$|\Theta (Y_1,\tau)-\bar V|\leq \delta^4 \ \ \ \forall \tau\in [(\delta \gamma)^2,\delta^{-2}].$$
Moreover, since $Y_1$ is $(j,\Psi(\zeta),1)$-cylindrical and $Y_{i^1} \in \pball_{\ep} (L_{Y_1})$, we know that each $Y_{i^1}$ satisfies
$$|\Theta(Y_1,\tau)-\Theta (Y_{i^1},\tau)|\leq \delta^4, \ \ \ \forall \tau\in [(\delta \gamma)^2,\delta^{-2}],$$
if $\ep$ and $\zeta$ are sufficiently small. Combining the above two inequalities, we get
\begin{align*}
|\Theta (Y_{i^1},\tau)-\bar V|\leq \delta^3, \ \ \ \forall \tau\in [(\delta \gamma)^2,\delta^{-2}],
\end{align*}
which establishes $(n2)$. Property $(n3)$ also follows from our construction and Theorem \ref{quantitativeuniq}. Finally, $(n4)$ follows from Lemma \ref{lem:x01} and our construction, provided that $\ep$ and $\zeta$ are sufficiently small.

We can now reorganize this cover according to the types of balls:
$$\pball(Y,2)\subset \nnn^1\bigcup\bigcup_{b^1} \pball(Y_{b^1}, \gamma)\bigcup\bigcup_{c^1}\pball(Y_{c^1},\gamma)\bigcup\bigcup_{d^1}\pball(Y_{d^1}, \gamma)\bigcup\bigcup_{e^1}\pball(Y_{e^1},\gamma).$$
In addition, we have the following content estimate from Proposition \ref{prop:ahlfors}:
$$\sum_{b^1}\gamma^j+\sum_{c^1}\gamma^j+\sum_{d^1}\gamma^j+\sum_{e^1}\gamma^j\leq C(n).$$

Now for each $c$-ball $\pball(Y_{c^1},\gamma)$, we fix a point $Y_{c^1,1} \in \mathcal{V}_{\zeta, \gamma}(Y_{c^1})$, which is $(j, \Psi(\zeta),\gamma)$-cylindrical with respect to $L_{Y_{c^1,1}}$. Since $Y_{c^1}$ is $(j,\Psi(\zeta),1)$-cylindrical, we have
\begin{align} \label{eq:cons3}
L_{Y_{c^1}}\subset \pball_{\ep} (L_{Y_{c^1,1}}).
\end{align}

Now, we repeat the above argument and choose a Vitali cover $\{\pball(Y_{i^2},c_n^2\gamma)\}$ of
$$ \bigcup_{c^1} \lc L_{c^1, 1} \bigcap\pball(Y_{c^1}, 4 \gamma) \rc \setminus \lc \bigcup_{b^1} \pball(Y_{b^1}, \gamma)\bigcup\bigcup_{d^1}\pball(Y_{d^1},\gamma)\bigcup_{e^1}\pball(Y_{e^1},\gamma) \rc $$
such that $\{\pball(Y_{i^2},c_n^3\gamma)\}$ are pairwise disjoint and $Y_{i^2}\in\MM\bigcap \pball_{\ep \gamma}\lc \bigcup_{c^1} ( L_{c^1, 1} \bigcap\pball(Y_{c^1}, 4 \gamma))  \rc$. Following the same reasoning as above and using \eqref{eq:cons3}, we obtain
\begin{align*}
|\Theta(Y_{i^2},\tau)-\bar V|<\delta^3, \ \ \ \forall \tau\in [ (\delta \gamma^2)^2,\delta^{-2}].
\end{align*}

We can reindex the balls according to their types. Define
$$\nnn^2 := \pball(Y,2)\setminus \lc \bigcup_{c^2}\pball(Y_{c^2},\gamma^2)\bigcup\bigcup_{1\leq k\leq 2}(\bigcup_{b^k}\pball(Y_{b^k},\gamma^k)\bigcup_{d^k}\pball(Y_{d^k},\gamma^k)\bigcup\bigcup_{e^k} \pball(Y_{e^k},\gamma^k) \rc.$$
Then, we can rewrite the decomposition as 
$$\pball(Y,2)\subset \nnn^2\bigcup\bigcup_{c^2}\pball(Y_{c^2},\gamma^2)\bigcup\bigcup_{1\leq k\leq 2} \lc \bigcup_{b^k}\pball(Y_{b^k},\gamma^k)\bigcup_{d^k}\pball(Y_{d^k},\gamma^k)\bigcup\bigcup_{e^k} \pball(Y_{e^k},\gamma^k) \rc.$$

As reasoned above, $\nnn^2$ is a neck region with a center consisting of the centers of the above balls and the radius function chosen to match the corresponding radii of these balls. By Proposition \ref{prop:ahlfors}, the following content estimate holds:
$$\sum_{1\leq k\leq 2} \lc \sum_{b^k}(\gamma^k)^j+\sum_{d^k}(\gamma^k)^j+\sum_{e^k}(\gamma^k)^j \rc+\sum_{c^2}(\gamma^2)^j \leq C(n).$$

Repeating the above decomposition for $l$-steps, we obtain the neck region:
denote by $$\nnn^l:=\pball(Y,2)\setminus \lc \bigcup_{c^l}\pball(Y_{c^l},\gamma^l)\bigcup\bigcup_{1\leq k\leq l}(\bigcup_{b^k}\pball(Y_{b^k},\gamma^k)\bigcup_{d^k}\pball(Y_{d^k},\gamma^k)\bigcup\bigcup_{e^k} \pball(Y_{e^k},\gamma^k) \rc,$$
together with the decomposition:
$$\pball(Y,2)\subset \nnn^l\bigcup\bigcup_{c^l}\pball(Y_{c^l},\gamma^l)\bigcup\bigcup_{1\leq k\leq l} \lc \bigcup_{b^k}\pball(Y_{b^k},\gamma^k)\bigcup_{d^k}\pball(Y_{d^k},\gamma^k)\bigcup\bigcup_{e^k} \pball(Y_{e^k},\gamma^k) \rc.$$
At each step, we have 
$$|\Theta(Y_{i^l},\tau)-\bar V|<\delta^3, \quad \forall \tau\in [(\delta \gamma^l)^2,\delta^{-2}].$$
Additionally, by Proposition \ref{prop:ahlfors}, the following content estimate holds:
\begin{align} \label{eq:cons5}
\sum_{1\leq k\leq l} \lc \sum_{b^k}(\gamma^k)^j+\sum_{d^k}(\gamma^k)^j+\sum_{e^k}(\gamma^k)^j \rc+\sum_{c^l}(\gamma^l)^j\leq C(n).
\end{align}

Now, we set $\ccc^l :=\{Y_{c^l}\}$. Then, by \eqref{eq:cons5}, we have
\begin{align} \label{eq:cons5a}
\vol(\pball_{4\gamma^l}(\ccc^l)) \le C(n) \sum_{c^l}(\gamma^l)^{n+3} \leq C(n) \gamma^{l(n+3-j)}.
\end{align}
It is clear from the construction that $\ccc^{l+1} \subset \pball_{4\gamma^l}(\ccc^l)$. We denote the Hausdorff limit of $\ccc^l$ by $\ccc_0$. Then, by \eqref{eq:cons5a}, we have
\begin{align}\label{eq:cons5b}
\vol(\pball_{4\gamma^l}(\ccc_0)) \le C(n) \gamma^{l(n+3-j)}.
\end{align}
and hence $\mathscr{H}_P^j(\ccc_0) \le C(n)$.

Moreover, we have
$$\pball(Y,2)\subset (\nnn\bigcup\ccc_0)\bigcup\bigcup_{1\leq k<\infty}\lc \bigcup_{b^k}\pball(Y_{b^k},\gamma^k)\bigcup_{d^k}\pball(Y_{d^k},\gamma^k)\bigcup\bigcup_{e^k} \pball(Y_{e^k},\gamma^k) \rc$$
where
$$\nnn := \pball(Y,2)\setminus \lc \ccc_0\bigcup\bigcup_{1\leq k<\infty}(\bigcup_{b^k}\pball(Y_{b^k},\gamma^k)\bigcup_{d^k}\pball(Y_{d^k},\gamma^k)\bigcup\bigcup_{e^k} \pball(Y_{e^k},\gamma^k) \rc.$$

From our construction, it follows that $\nnn$ is a $(j,\delta, 1)$-neck region. Additionally, by Proposition \ref{prop:ahlfors}, the following content estimate holds:
$$\sum_{1\leq k<\infty} \lc \sum_{b^k}(\gamma^k)^j+\sum_{d^k}(\gamma^k)^j+\sum_{e^k}(\gamma^k)^j \rc+\mathscr{H}_P^j(\ccc_0)\leq C(n).$$
In summary, we have completed the proof.
\end{proof}

\begin{prop}[Decomposition of $d$-balls]\label{decompositiondball}
For $Y\in\MM,s>0$ with $\tf(Y)>4\zeta^{-2}s^2$, let $\pball(Y,s)$ be a $d$-ball with $\pball(Y,4s)\subset \pball(X,4)$. If $\delta\leq\delta(n,\alpha,\Lambda)$, $\beta \le \beta(n)$ and $\zeta\leq\zeta(n,\alpha,\Lambda,\delta,\beta,\eta)$, then
\begin{equation*}
\pball(Y,s) \bigcap \MM \subset \bigcup_b \pball(Z_b,r_b)\bigcup\bigcup_c\pball(Z_c, r_c)\bigcup\bigcup_e\pball(Z_e,r_e)\bigcup\tilde S
\end{equation*}
satisfying
\begin{enumerate} [label=(\roman*)]
\item for each $b$, $\pball(Z_b,r_b)$ is a $b$-ball;
\item for each $c$, $\pball(Z_c,r_c)$ is a $c$-ball;
\item for each $e$, $\pball(Z_e,r_e)$ is an $e$-ball;
\item $\tilde S \subset \MS^{j-1}$ with $\mathscr{H}_P^j(\tilde S)=0$;
\item the following content estimates hold:
\begin{align*}
\sum_b r_b^j+\sum_e r_e^j\leq C(n,\beta)s^j \quad \text{and} \quad \sum_c r_c^j\leq C(n)\beta s^j.
\end{align*}
\end{enumerate}
\end{prop}

\begin{proof}
Assume $s=1$ and choose a Vitali cover $\{\pball(Y_{i^1},\beta)\}_{Y_{i^1}\in \pball(Y,1)\bigcap\MM}$ of $\pball(Y,1)\bigcap\MM$ such that $\{\pball(Y_{i^1}, \beta/5)\}$ are pairwisely disjoint. We reorganize the cover by types:
$$\pball(Y,1) \bigcap\MM \subset\bigcup_{b^1} \pball(Y_{b^1},\beta)\bigcup\bigcup_{c^1}\pball(Y_{c^1},\beta)\bigcup\bigcup_{d^1}\pball(Y_{d^1},\beta)\bigcup\bigcup_{e^1} \pball(Y_{e^1}, \beta).$$
Using disjointness and comparing the volumes, we have
$$\sum_{b^1} \beta^j+\sum_{c^1}\beta^j+\sum_{d^1}\beta^j+\sum_{e^1}\beta^j\leq C(n)\beta^{j-n-3}=C(n,\beta).$$
Moreover, since $\pball(Y,1)$ is a $d$-ball, we know $\mathcal{V}_{\zeta,1}(Y)$ is nonempty and not $(j,\beta,4)$-independent. Since each point $Z \in \mathcal{V}_{\zeta,1}(Y)$ is $(j, \Psi(\zeta),1)$-cylindrical, it follows from Definition \ref{independentpoints} that there exists a $j-1$ affine plane $V \subset \R^{n+1}$ such that
\begin{align*}
\mathcal{V}_{\zeta,1}(Y)\subset \pball_{5\beta}(L),
\end{align*}
where $L:=V \times \{\tf(Y)\}$. Since $\mathcal V_{\zeta,\beta}(Y_{c^1}) \bigcup \mathcal V_{\zeta,\beta}(Y_{d^1}) \subset \mathcal{V}_{\zeta,1}(Y)$, we conclude
\begin{align*}
\bigcup_{c^1}\pball(Y_{c^1},\beta)\bigcup\bigcup_{d^1}\pball(Y_{d^1},\beta)\subset \pball_{8 \beta}(L).
\end{align*}

Let $K$ denote the total number of all $c$-balls and $d$-balls. Since the total parabolic volume of $\pball_\beta(L)\bigcap \pball (Y,1)$ is less than $C(n)\beta^{n+3-(j-1)}$, by the Vitali covering, we have
$$K\leq C(n)\beta^{n+3-(j-1)}\beta^{-n-3}=C(n)\beta^{-j+1}.$$
Thus, we obtain a refined estimate for all $c$-balls and $d$-balls:
\begin{equation*}
\sum_{c^1} \beta^j+\sum_{d^1}\beta^j\leq C(n)\beta^{-j+1}\beta^j=C(n)\beta.
\end{equation*}
For each $d$-ball $\pball(Y_{d^1},\beta)$, we choose a Vitali cover $\{\pball(Y_{i^2},\beta^2)\}_{Y_{i^2}\in \pball(Y_{d^1},r_{d^1})\bigcap\MM}$ of $\pball(Y_{d^1},r_{d^1})\bigcap\MM$ with $\{\pball(Y_{i^2},\beta^2/5)\}$ pairwise disjoint. We can rewrite this decomposition by types:
$$\pball(Y_{d^1},\beta)\bigcap\MM \subset\bigcup_{b^2} \pball(Y_{b^2},\beta^2)\bigcup\bigcup_{c^2}\pball(Y_{c^2},\beta^2)\bigcup\bigcup_{d^2}\pball(Y_{d^2},\beta^2)\bigcup\bigcup_{e^2} \pball(Y_{e^2},\beta^2).$$
Combining with the previous decomposition, we have:
\begin{equation*}
\pball(Y,1) \bigcap\MM\subset\bigcup_{d^2}\pball(Y_{d^2},\beta^2)\bigcup\bigcup_{1\leq k\leq 2} \lc \bigcup_{b^k}\pball(Y_{b^k},\beta^k)\bigcup\bigcup_{c^k}\pball(Y_{c^k},\beta^k)\bigcup\bigcup_{e^k}\pball(Y_{e^k},\beta^k) \rc,
\end{equation*}
with the content estimates:
\begin{align*}
&\sum_{1\leq k\leq 2}(\sum_{b^k}(\beta^k)^j+\sum_{e^k}(\beta^k)^j)\leq C(n,\beta)+C(n,\beta)(C(n)\beta)=C(n,\beta)\sum_{1\leq k\leq 2}(C(n)\beta)^{k-1};\\
&\sum_{1\leq k\leq 2}\sum_{c^k} (\beta^k)^j\leq C(n)\beta+C(n)\beta(C(n)\beta)=\sum_{1\leq k\leq 2}(C(n)\beta)^k;\\
&\sum_{d^2} (\beta^2)^j\leq (C(n)\beta)^2.
\end{align*}
Now, we repeat the above decomposition iteratively for $d$-balls. At the $l$-th step, we obtain the decomposition:
\begin{equation*}
\pball(Y,1) \bigcap\MM \subset\bigcup_{d^l}\pball(Y_{d^l},\beta^l)\bigcup\bigcup_{1\leq k\leq l} \lc \bigcup_{b^k}\pball(Y_{b^k},\beta^k)\bigcup\bigcup_{c^k}\pball(Y_{c^k},\beta^k)\bigcup\bigcup_{e^k}\pball(Y_{e^k},\beta^k) \rc,
\end{equation*}
with the content estimates:
\begin{align*}
&\sum_{1\leq k\leq l}(\sum_{b^k}(\beta^k)^j+\sum_{e^k}(\beta^k)^j)\leq C(n,\beta)\sum_{1\leq k\leq l}(C(n)\beta)^{k-1};\\
&\sum_{1\leq k\leq l}\sum_{c^k} (\beta^k)^j\leq \sum_{1\leq k\leq l}(C(n)\beta)^k;\\
&\sum_{d^l} (\beta^l)^j\leq (C(n)\beta)^l.
\end{align*}
Here, we assume $\beta$ is small enough such that $C(n)\beta\leq 1/2$, ensuring the convergence of the above estimates.

Let $\tilde S^l :=\{Y_{d^l}\}$. By our construction, it is clear that $\tilde S^{l+1}\subset \pball_{2\beta^l}(\tilde S^l)$. We denote the Hausdorff limit of $\tilde S^l$ by $\tilde S$. Then, we obtain the decomposition:
\begin{equation*}
\pball(Y,1) \bigcap\MM \subset\tilde S\bigcup\bigcup_{1\leq k<\infty} \lc \bigcup_{b^k}\pball(Y_{b^k},\beta^k)\bigcup\bigcup_{c^k}\pball(Y_{c^k},\beta^k)\bigcup\bigcup_{e^k}\pball(Y_{e^k},\beta^k) \rc
\end{equation*}
with
\begin{align*}
\sum_{1\leq k<\infty}\lc \sum_{b^k}(\beta^k)^j+\sum_{e^k}(\beta^k)^j \rc \leq C(n,\beta) \quad \text{and} \quad \sum_{1\leq k<\infty}\sum_{c^k} (\beta^k)^j\leq C(n)\beta.
\end{align*}

Finally, we show that $\tilde S\subset \MS^{j-1}$. For any $Z\in \tilde S$, choose $Y^l\in \{Y_{d^l}\}$ such that $Y^l\to Z$ in the Hausdorff sense. From our construction, we have
\begin{equation*}
\rVert Z-Y^l\rVert\leq\sum_{k=l}^\infty 2\beta^k\leq 3\beta^l.
\end{equation*}
Since $\pball(Y^l,\beta^l)$ is a $d$-ball, there exists $Z^l \in \pball(Y^l, 4\beta^l) \bigcap \MM$ such that 
\begin{equation}\label{eq:ex001}
\Theta(Z^l, (\zeta \beta^l)^2)>\bar V-\zeta/2.
\end{equation}
Additionally, we have
\begin{equation*}
\rVert Z-Z^l\rVert\leq\rVert Z-Y^l\rVert+\rVert Y^l-Z^l\rVert\leq 7\beta^l.
\end{equation*}
Now, suppose $Z$ has a tangent flow given by a cylinder in $\CC^k$ for $k \ge j$. Then, $Z$ is $(k, \Psi(l^{-1}), \beta^l)$-cylindrical with respect to $L_Z$, and moreover, $\|Z^l-L_Z\| \le \Psi(\zeta+l^{-1}) \beta^l$. It follows from \eqref{eq:ex001} that if $l$ is sufficiently large, 
\begin{equation*}
\Theta(Y, (\zeta \beta^l)^2)>\bar V-\zeta
\end{equation*}
for any $Y \in \pball(Y^l, 4\beta^l) \bigcap \pball_{\Psi(\zeta+l^{-1}) \beta^l}(L_Z) \bigcap \MM$. Moreover, for any $Y' \in \pball(Y^l, 4\beta^l) \bigcap L_Z$, there exists $Y'' \in \MM$ such that $\|Y'-Y''\| \le \Psi(\zeta+l^{-1}) \beta^l$. However, this leads to a contradiction, since $\VV_{\zeta, \beta^l}(Y^l)$ is not $(j, \beta,4 \beta^l )$-independent in $\pball(Y^l, 4 \beta^l)$.

Therefore, $\tilde S \subset \MS^{j-1}$. Furthermore, it is clear that $\mathscr{H}_P^j(\tilde S)=0$, as the parabolic dimension of $\MS^{j-1}$ is at most $j-1$.
\end{proof}

Combining Proposition \ref{decompositioncball} and Proposition \ref{decompositiondball}, we can now prove the inductive decomposition:

\begin{prop}[Inductive decomposition]\label{inductivedecomposition}
	Let $Y\in\MM, s>0$ with $\tf(Y)>4\zeta^{-2}s^2$. 
If $\delta\leq\delta(n,\alpha,\Lambda)$, $\beta \le \beta(n)$ and $\zeta\leq\zeta(n,\alpha,\Lambda,\delta,\beta,\eta)$, then we have the following decomposition
$$\pball (Y,s) \bigcap \MM \subset\bigcup_a \lc (\nnn_a\bigcup \ccc_{0,a})\bigcap\pball(Y_a,r_a) \rc \bigcup \bigcup_b \pball (Y_b,r_b)\bigcup\bigcup_e\pball(Y_e,r_e)\bigcup S_0,$$
satisfying
\begin{enumerate} [label=(\roman*)]
\item for all $a$, $\nnn_a=\pball(Y_a, 2r_a)\setminus \pball_{r_{a,Y}}(\ccc_a)$ is a $(j,\delta, r_a)$-neck region;
\item for each $b$, $\pball(Z_b,r_b)$ is a $b$-ball;
\item for each $e$, $\pball(Z_e,r_e)$ is an $e$-ball;
\item $S_0 \subset \MS^{j-1}$ with $\mathscr{H}_P^j(S_0)=0$;
\item the following content estimates hold:
\begin{align*}
\sum_a r_a^j+\sum_b r_b^j+\sum_e r_e^j+\mathscr{H}_P^j(\bigcup_a\ccc_{0,a})\leq C(n) s^j.
\end{align*}
\end{enumerate}
\end{prop}

\begin{proof}
We assume $s=1$. Suppose $\pball(Y,1)$ is not a $b$-ball or $e$-ball; otherwise, we are done. If $\pball(Y,1)$ is a $d$-ball, then we apply Proposition \ref{decompositiondball} to get
$$\pball(Y,1) \bigcap \MM \subset\bigcup_{b'} \pball(Y_{b'},r_{b'})\bigcup\bigcup_{c'}\pball(Y_{c'},r_{c'})\bigcup\bigcup_{e'} \pball(Y_{e'},r_{e'})\bigcup\tilde S^1,$$
with $\tilde S^1 \subset \MS^{j-1}$ and the following content estimates:
\begin{align*}
\sum_{b'} (r_{b'})^j+\sum_{e'} (r_{e'})^j&\leq C(n),\\
\sum_{c'}(r_{c'})^j&\leq C(n)\beta.
\end{align*}
Now, for each $c$-ball $\pball(Y_{c'},r_{c'})$, we apply Proposition \ref{decompositioncball} to obtain
$$\pball(Y_{c'},r_{c'})\subset \lc (\ccc_{0}\bigcup \nnn) \bigcap \pball(Y_{c'},r_{c'}) \rc\bigcup\bigcup_{b''}\pball(Y_{b''},r_{b''})\bigcup\bigcup_{d''}\pball (Y_{d''},r_{d''})\bigcup\bigcup_{e''} \pball(Y_{e''},r_{e''}),$$
with the content estimate:
\begin{align*}
\sum_{b''} (r_{b''})^j+\sum_{d''}(r_{d''})^j+\sum_{e''} (r_{e''})^j+\mathscr{H}_P^j(\ccc_0)\leq C(n)(r_{c'})^j.
\end{align*}
Combining the above two decompositions and reindexing the balls, we obtain
\begin{align*}
\pball (Y,1) \bigcap \MM \subset &\bigcup_{a^1}\lc \ccc_{0, a^1}\bigcup \nnn_{a^1} \bigcap \pball(Y_{a^1},r_{a^1})\rc\bigcup\bigcup_{b^1} \pball(Y_{b^1},r_{b^1})\bigcup\bigcup_{d^1}\pball(Y_{d^1},r_{d^1}) \\
&\bigcup\bigcup_{e^1}\pball(Y_{e^1},r_{e^1})\bigcup \tilde S^1,
\end{align*}
with 
\begin{align*}
&\sum_{b^1} (r_{b^1})^j+\sum_{e^1} (r_{e^1})^j+\leq C(n),\\
&\sum_{a^1} \lc (r_{a^1})^j+\mathscr{H}_P^j(\ccc_{0,a^1}) \rc +\sum_{d^1} (r_{d^1})^j\leq C(n)\beta.
\end{align*}
Next, we apply the above decomposition to each $d$-ball $\pball(Y_{d^1},r_{d^1})$ to get the second decomposition:
\begin{align*}
\pball (Y,1) \bigcap \MM \subset &\bigcup_{1\leq k\leq 2} \lc (\bigcup_{a^k} (\ccc_{0,a^k}\bigcup \nnn_{a^k}\bigcap \pball (Y_{a^k},r_{a^k}))\bigcup\bigcup_{b^k} \pball(Y_{b^k},r_{b^k})\bigcup\bigcup_{e^k}\pball(Y_{e^k},r_{e^k})\bigcup \tilde S^k \rc \\&
\bigcup\bigcup_{d^2} \pball(Y_{d^2},r_{d^2}),
\end{align*}
with $\bigcup_{1\leq k\leq 2}  \tilde S^k \subset \MS^{j-1}$ and the content estimates:
\begin{align*}
\sum_{1\leq k\leq 2} \lc \sum_{a^k} (r_{a^k})^j+\sum_{b^k} (r_{b^k})^j+\sum_{e^k} (r_{e^k})^j+\sum_{a^k} \mathscr{H}_P^j(\ccc_{0,a^k}) \rc \leq C(n)+C(n)\beta,
\end{align*}
and
$$\sum_{d^2} (r_{d^2})^j\leq (C(n)\beta)^2.$$
Repeating this decomposition $l$ times, we get
\begin{align*}
\pball (Y,1)\bigcap \MM \subset &\bigcup_{1\leq k\leq l}\lc (\bigcup_{a^k}(\ccc_{0,a^k}\bigcup \nnn_{a^k}\bigcap \pball (Y_{a^k},r_{a^k}))\bigcup\bigcup_{b^k} \pball(Y_{b^k},r_{b^k})\bigcup\bigcup_{e^k}\pball(Y_{e^k},r_{e^k})\bigcup \tilde S^k \rc \\&
\bigcup\bigcup_{d^l} \pball(Y_{d^l},r_{d^l}),
\end{align*}
with $\bigcup_{1\leq k\leq l}  \tilde S^k \subset \MS^{j-1}$ and the content estimates:
\begin{align*}
\sum_{1\leq k\leq l} \lc \sum_{a^k} (r_{a^k})^j+\sum_{b^k} (r_{b^k})^j+\sum_{e^k} (r_{e^k})^j+\sum_{a^k} \mathscr{H}_P^j(\ccc_{0,a^k}) \rc \leq C(n)\sum_{j=1}^l(C(n)\beta)^{j-1}
\end{align*}
and
$$\sum_{d^l} (r_{d^l})^j\leq (C(n)\beta)^l.$$

We assume $\beta$ is small such that $C(n)\beta\leq 1/2$. Denote $\tilde S_{d^k}=\{Y_{d^k}\}$ and set $ (\tilde S_{d^k})_{10}:=\bigcup_{d^k} \pball_{10 r_{d^k}}(Y_{d^k})$. It is clear from our construction that if $\beta \le \beta(n)$, $(\tilde S_{d^{k+1}})_{10} \subset (\tilde S_{d^k})_{10}$. We set 
\begin{align*}
\tilde S_d :=\bigcap_{k\geq 1} (\tilde S_{d^k})_{10}
\end{align*}
By the same reasoning as in the proof of Proposition \ref{decompositiondball}, we conclude that $\tilde S_d \subset \MS^{j-1}$.

Now, define $S_0 := \tilde S_d\bigcup\bigcup_{k\geq 1}\tilde S^k$. It follows that $S_0\subset \MS^{j-1}$, and consequently, $\mathscr{H}_P^j(S_0)=0$. Therefore, the proof is complete in this case.

If $\pball (Y,1)$ is a $c$-ball, we first apply Proposition \ref{decompositioncball} to obtain a decomposition consisting only of $b$-balls, $d$-balls, $e$-balls and neck regions. Then, we can repeat the argument above to each $d$-ball to complete the proof.
\end{proof}

With all these preparations, we can prove Theorem \ref{neckdecomposition}.

\emph{Proof of Theorem \ref{neckdecomposition}}: We assume $r=1$. In the proof, we first fix $\beta=\beta(n)$ and then choose $\zeta=\zeta(n,\alpha,\Lambda,\delta,\beta,\eta)=\zeta(n,\alpha,\Lambda,\delta,\eta)$ so that Proposition \ref{inductivedecomposition} holds. 

Next, we choose a Vitali cover $\{\pball(Y_i, \zeta)\}_{Y_i \in \pball(X,1) \bigcap \MM} $ of $\pball(X,1) \bigcap \MM$ such that $\{\pball (Y_i, \zeta/5)\}$ are pairwisely disjoint. Note that for each $i$, 
\begin{align} \label{eq:eballes}
\bar V_{Y_i,\zeta}(\zeta)=\sup_{Z\in \pball (Y_i,4\zeta)}\Theta (Z, (\zeta^{-1}\cdot \zeta)^2)\leq \sup_{Z\in \pball(X,4)}\Theta (Z,1)=\bar V_{X,1}(1).
\end{align}

Then we apply Proposition \ref{inductivedecomposition} to each $\pball(Y_i,\zeta)$, with $\bar V$ being $ \bar V_{Y_i,\zeta}(\zeta)$, to get 
\begin{equation}\label{decom1}
	\pball(X,1) \bigcap \MM \subset\bigcup_a \lc (\nnn_a^1\bigcup\ccc_{0,a}^1)\bigcap \pball(Y_a^1,r_a^1)\rc\bigcup\bigcup_b \pball(Y_b^1,r_b^1)\bigcup\bigcup_{e'}\pball(Y_{e'}^1,r_{e'}^1)\bigcup S_0^1,
\end{equation}
with $r_a^1,r_b^1,r_{e'}^1\leq\zeta$, $S_0^1 \subset \MS^{j-1}$ and the content estimate:
\begin{equation*}
\sum_a (r_a^1)^j+\sum_b (r_b^1)^j+\sum_e (r_{e'}^1)^j+\mathscr{H}_P^j(\bigcup_a\ccc_{0,a}^1)\leq C(n,\alpha,\Lambda,\delta,\eta).
\end{equation*} 

Note that for each $e$-ball $\pball(Y^1_{e'},r_{e'}^1)$, we have
\begin{equation}\label{eballesti1}
\bar V_{Y^1_{e'}, r_{e'}^1}(\zeta^{-1})=\sup_{Z\in \pball(Y^1_{e'},4r_{e'}^1)}\Theta (Z,(\zeta r_{e'}^1)^2)\leq \sup_i \{\bar V_{Y_i,\zeta}(\zeta)\}-\zeta/2 \leq\bar V_{X,1}(1)-\zeta/2,
\end{equation} 
where, for the second inequality, we have used the definition of the $e$-ball with respect to $\bar V_{Y_i,\zeta}(\zeta)$, and the last inequality follows from \eqref{eq:eballes}. For each $\pball(Y^1_{e'},r_{e'}^1)\cap \MM$, consider a Vitali covering $\{\pball(Y^1_{e''},r_{e''}^1)\}$ with $Y^1_{e''}\in \pball(Y^1_{e'},r_{e'}^1)\cap \MM, r_{e''}^1=\zeta r_{e'}^1$ such that $\pball(Y^1_{e''},r_{e''}^1/5)$ are disjoint. By \eqref{eballesti1}, we have
\begin{equation}\label{eballesti2}
	\bar V_{Y^1_{e''}, r_{e''}^1}(1)=\bar V_{Y^1_{e'}, r_{e'}^1}(\zeta^{-1})\leq\bar V_{X,1}(1)-\zeta/2.
\end{equation}
Combining with the decomposition \eqref{decom1}, we can obtain:
\begin{equation*}
	\pball(X,1) \bigcap \MM \subset\bigcup_a \lc (\nnn_a^1\bigcup\ccc_{0,a}^1)\bigcap \pball(Y_a^1,r_a^1)\rc\bigcup\bigcup_b \pball(Y_b^1,r_b^1)\bigcup\bigcup_{e}\pball(Y_{e}^1,r_{e}^1)\bigcup S_0^1,
\end{equation*}
with $r_a^1,r_b^1,r_{e}^1\leq\zeta$, $S_0^1 \subset \MS^{j-1}$ and the content estimate:
\begin{equation*}
	\sum_a (r_a^1)^j+\sum_b (r_b^1)^j+\sum_e (r_{e}^1)^j+\mathscr{H}_P^j(\bigcup_a\ccc_{0,a}^1)\leq C(n,\alpha,\Lambda,\delta,\eta).
\end{equation*}
Moreover, for each $e$-ball $\pball(Y_{e}^1,r_{e}^1)$, it follows from \eqref{eballesti2} that
\begin{equation*}
	\bar V_{Y^1_{e}, r_{e}^1}(1)\leq\bar V_{X,1}(1)-\zeta/2.
\end{equation*}

Applying Proposition \ref{inductivedecomposition} to $\pball(Y_e^1, r_e^1)$ and repeating the above process, we obtain the following decomposition
$$\pball(Y_e^1,r_e^1)\bigcap \MM \subset\bigcup_a \lc (\nnn_a^2\bigcup\ccc_{0,a}^2)\bigcap \pball(Y_a^2,r_a^2) \rc\bigcup\bigcup_b \pball(Y_b^2,r_b^2)\bigcup\bigcup_e\pball(Y_e^2,r_e^2)\bigcup S_0^2$$
with $S_0^2 \subset \MS^{j-1}$ and the content estimate:
\begin{equation*}
\sum_a (r_a^2)^j+\sum_b (r_b^2)^j+\sum_e (r_e^2)^j+\mathscr{H}_P^j(\bigcup_a\ccc_{0,a}^2)\leq C(n,\alpha,\Lambda,\delta,\eta)(r_e^1)^j. 
\end{equation*}
Moreover, we have $\bar V_{Y_e^2,r_e^2}(1) \leq \bar V_{Y_e^1,r_e^1}(1)-\zeta/2 \leq\bar V_{X,1}(1)-\zeta$. 

We can now apply Proposition \ref{inductivedecomposition} again to $\pball(Y_e^2,r_e^2)$ to continue the decomposition. At each step, the corresponding quantity $\bar V_{\ast,\ast}(1)$ decreases strictly by $\zeta/2$. Thus, the process must terminate in $\frac{2\Lambda}{\zeta}+1$ steps, meaning that, after the last decomposition, no $e$-balls remain. We denote
$$S_0 := \bigcup_i S_0^i.$$
After reindexing the balls, we obtain the following decomposition:
\begin{equation} \label{eq:decom}
\pball(X, 1) \bigcap \MM \subset \bigcup_a \lc (\nnn_a\bigcup \ccc_{0,a} )\bigcap \pball(Y_a,r_a) \rc \bigcup_b \pball(Y_b,r_b)\bigcup S_0
\end{equation}
$S_0 \subset \MS^{j-1}$ and the content estimate:
\begin{equation*}
\sum_a r_a^j+\sum_b r_b^j+\mathscr{H}_P^j(\bigcup_a\ccc_{0,a})\leq C(n,\alpha,\Lambda,\delta,\eta);
\end{equation*}
Combined with Theorem \ref{neckstructurethm}, we have shown that (a), (b), (c), and (d) in Theorem \ref{neckdecomposition} hold. 

With \eqref{eq:decom}, to prove (e), it suffices to show that if $Y\in \MM \bigcap \nnn_a$ or $Y\in \MM \bigcap \pball(Y_b,r_b)$, then $Y\notin \mathcal{S}^j_{\epsilon}$, provided that $\eta$ and $\delta$ are sufficiently small.

\textbf{Case 1}: Suppose $Y\in \MM\bigcap\nnn_a$ and set $r=\|Y-\ccc_a\|$. It follows from Lemma \ref{lem:regu} that $Y$ is $(n,\ep, \chi r)$-cylindrical for $\chi=\chi(n, \ep)$, and hence $Y\notin \mathcal{S}^j_{\epsilon}$ if $\delta \le \delta(n, \alpha, \Lambda, \ep)$.

\textbf{Case 2}: If $Y\in \MM\bigcap\pball (Y_b, r_b)$, by the definition of $b$-balls, there exists $Y_b' \in \pball(Y_b, 4r_b)$ such that $Y'_b$ is $(j', \eta, r_b)$-cylindrical for $j'>j$. For any $\ep'>0$, there exists $\eta \le \eta(n, \ep')$ such that if $\|Y- L_{Y_b'} \| \le \ep' r_b$, then $Y$ is $(j', \Psi(\ep'), r_b)$-cylindrical, which implies $Y\notin \mathcal{S}^j_{\epsilon}$ if $\ep'$ is sufficiently small. Alternatively, if $\|Y- L_{Y_b'}\| \ge \ep' r_b$, we conclude that $Y$ is $(n, \Psi(\eta), \Psi(\eta) \ep' r_b)$-cylindrical, which also means $Y\notin \mathcal{S}^j_{\epsilon}$.

In sum, we have completed the proof of Theorem \ref{neckdecomposition}.

\section{Proof of Theorem \ref{volumeestimate1}}
In this section, we prove Theorem \ref{volumeestimate1}.

We first fix all constants. Given $\ep>0$, $\alpha>0$ and $\Lambda>0$, we choose small positive constants $\eta=\eta(n,\ep)$, $\delta=\delta(n, \alpha, \Lambda, \ep)$, $\zeta=\zeta(n, \alpha, \Lambda, \ep)$ and $\chi=\chi(n, \ep)$ such that the following properties hold:
\begin{enumerate} [label=(\roman*)]
\item If $\nnn_a=\pball(Y_a,2r_a)\setminus \pball_{r_{a,Y}}(\ccc_a)$ is a $(j,\delta,r_a)$-neck region, then any $Y \in \nnn_a \bigcap \MM$ is $(n, \ep, s)$-cylindrical for any $s \le \chi\|Y - \ccc_a\|$.

\item If $\pball(Y_b,r_b)$ is a $b$-ball, then any $Y \in \pball(Y_b,2r_b) \bigcap \MM$ is $(j', \ep, \chi r_b)$-cylindrical for some $j'>j$.

\item Theorem \ref{neckdecomposition} holds for the constants $\eta$, $\delta$ and $\zeta$.
\end{enumerate}

Indeed, (i) above follows from Lemma \ref{lem:regu}, and (ii) can be deduced similarly to Case 2 in the proof of Theorem \ref{neckdecomposition}.

Without loss of generality, we may assume $\tf(X)-16\zeta^{-2}>0$; otherwise, we can cover $\pball(X, 1)$ by $\{\pball(X_i, \zeta/4)\}$ and prove the conclusion for each ball individually. By Theorem \ref{neckdecomposition}, we have
\begin{equation} \label{eq:main1}
\pball(X, 2)\subset \bigcup_a \lc (\nnn_a\bigcup \ccc_{0,a}) \bigcap \pball(Y_a,r_a) \rc \bigcup_b \pball(Y_b,r_b)\bigcup S_0
\end{equation}
with the content estimate:
\begin{equation} \label{eq:main2}
\sum_a r_a^j+\sum_b r_b^j+\mathscr{H}_P^j(\bigcup_a\ccc_{0,a})\leq C(n,\alpha,\Lambda,\ep).
\end{equation}

Since $r<1$, it is clear that $\pball_r(\MS_{\epsilon,r}^j) \bigcap \pball(X, 1) \subset \pball_r \lc \MS_{\epsilon,r}^j \bigcap \pball(X, 2) \rc$. Based on the decomposition \eqref{eq:main1}, we divide the control of the volume into two steps.

\textbf{Step 1}: For any $Y \in \MS_{\epsilon,r}^j \bigcap \nnn_a \bigcap \pball(Y_a,r_a)$, it follows from (i) above that $\|Y- \ccc_a\| < 2\chi^{-1} r$; otherwise, $Y$ would be $(n, \ep, 2r)$-cylindrical, which contradicts $Y \in \MS_{\epsilon,r}^j$.

If $r_a \ge 2 \chi^{-1} r$, we set $W=\{Z \in \ccc_a \mid \|Y-Z\|=\|Y-\ccc_a\| \text{ for some } Y \in \MS_{\epsilon,r}^j \bigcap \nnn_a \bigcap \pball(Y_a,r_a)\}$. Define $s=10\chi^{-1} r$. Then it follows from Vitali's covering theorem that there exist $Z_i \in W$ for $1 \le i \le K_a$ such that $\{\pball(Z_i,s)\}_{1\leq i\leq K_a}$ covers $W$ and $\{\pball(Z_i,s/5)\}_{1\leq i\leq K_a}$ are pairwise disjoint. 

Denote the corresponding packing measure on $\nnn_a$ by $\mu_a$ (cf. \eqref{eq:pack}). Then, by Proposition \ref{prop:ahlfors}, we have
\begin{align*}
K_aC^{-1}(n)s^j\leq \sum_{i=1}^{K_a} \mu_a(\pball(Z_i,s/5)) \le \mu_a (\pball(Y_a,2r_a))\leq C(n)r_a^j.
\end{align*}
Note that the assumptions of Proposition \ref{prop:ahlfors} are satisfied, as $r_{a,Z_i} \le s/5 \le r_a$. Thus, we conclude that $K_a \le C(n, \ep) r_a^jr^{-j}$. In addition, we have 
\begin{align*}
\MS_{\epsilon,r}^j \bigcap \nnn_a \bigcap \pball(Y_a,r_a) \subset \bigcup_i \pball(Z_i,2s)
\end{align*}
and hence,
\begin{align*}
\vol \lc \pball_r \lc \MS_{\epsilon,r}^j \bigcap \nnn_a \bigcap \pball(Y_a,r_a) \rc \rc \leq\sum_{i=1}^{K_a}\vol(\pball(Z_i, 3s) )\leq C(n, \ep) K_a r^{n+3}\leq C(n, \ep) r_a^jr^{n+3-j}.
\end{align*}
Set
\begin{align*}
\MS':=\pball_r \lc \MS_{\epsilon,r}^j \bigcap \bigcup_{r_a \ge 2\chi^{-1}r}( \nnn_a \bigcap \pball(Y_a,r_a)) \rc.
\end{align*}
In sum, we obtain
\begin{equation}\label{sub1}
\vol (\MS') \leq C(n, \ep) \sum_{a}r_a^jr^{n+3-j} \le C(n, \alpha, \Lambda,\ep) r^{n+3-j}, 
\end{equation}
where we have used \eqref{eq:main2} for the final inequality.

\textbf{Step 2}: We take a cover $\{\pball(Y_i,r)\}_{1\leq i\leq K}$ of $\MS_{\epsilon,r}^j \setminus \MS'$ such that $Y_i \in \MS_{\epsilon,r}^j \setminus \MS'$ and $\{\pball(Y_i,r/2)\}$ are pairwise disjoint. By our definition, if $\pball(Y_i,r/2) \bigcap \nnn_a \bigcap \pball(Y_a,r_a) \ne \emptyset$, then $r_a \le 2 \chi^{-1} r$ and hence $\pball(Y_a,r_a) \subset \pball(Y_i, 3\chi^{-1}r)$. Moreover, if $\pball(Y_i,r/2) \bigcap \pball(Y_b,r_b) \ne \emptyset$, then we must have $r_b \le 2\chi^{-1}r$; otherwise, $Y_i \in \pball(Y_b, 2r_b)$, and by (ii) above, $Y_i$ would be $(j', \ep, \chi r_b)$-cylindrical for some $j'>j$, which contradicts $Y_i \in \MS_{\epsilon,r}^j$. Therefore, we also have $\pball(Y_b,r_b) \subset \pball(Y_i, 3\chi^{-1}r)$.

Thus, it follows from \eqref{eq:main1} that
\begin{align*}
\pball(Y_i,r/2) \subset & \bigcup_{\pball(Y_a,r_a)\subset \pball(Y_i,3\chi^{-1}r)}\pball(Y_a,r_a)\bigcup\bigcup_{\pball(Y_b,r_b)\subset \pball(Y_i,3\chi^{-1}r)}\pball(Y_b,r_b)\bigcup (S_0\bigcup\bigcup_a \ccc_{0,a}).
\end{align*}
By comparing volumes, we obtain
\begin{align*}
C(n)^{-1}r^{n+3} \leq & \vol(\pball(Y_i,r/2))\\
\leq &\sum_{\pball(Y_a,r_a)\subset \pball(Y_i,3\chi^{-1}r)}\mathrm{Vol}(\pball(Y_a,r_a))+\sum_{\pball(Y_b,r_b)\subset \pball(Y_i,3\chi^{-1}r)}\vol(\pball(Y_b,r_b))\\
&+\mathscr{H}_P^{n+3}(S_0\bigcup\bigcup_a \ccc_{0,a}) \\
\leq & C(n) \lc \sum_{\pball(Y_a,r_a)\subset \pball(Y_i,3\chi^{-1}r)}r_a^{n+3}+\sum_{\pball(Y_b,r_b)\subset \pball(Y_i,3\chi^{-1}r)}r_b^{n+3} \rc\\
\leq & C(n,\ep)r^{n+3-j} \lc \sum_{\pball(Y_a,r_a)\subset \pball(Y_i,3\chi^{-1}r)}r_a^{j}+\sum_{\pball(Y_b,r_b)\subset \pball(Y_i,3\chi^{-1}r)}r_b^{j} \rc,
\end{align*}
where we have used $r_a\leq 2\chi^{-1}r, r_b\leq 2\chi^{-1}r$ and $\mathscr{H}_P^{n+3}(S_0\bigcup\bigcup_a \ccc_{0,a})=0$.

Taking summation over $i$ and using \eqref{eq:main2}, we get
\begin{align*}
KC(n)^{-1}r^{n+3}\leq C(n,\ep)r^{n+3-j} \lc \sum_a r_a^j+\sum_b r_b^j \rc \leq C(n,\alpha, \Lambda, \ep)r^{n+3-j}.
\end{align*}
and hence $K\leq C(n,\alpha, \Lambda, \ep) r^{-j}$. Since all $\{\pball (Y_i, 2r)\}$ cover $\pball_r(\MS_{\epsilon,r}^j \setminus \MS')$, we obtain
\begin{align}\label{sub2}
\vol(\pball_r(\MS_{\epsilon,r}^j \setminus \MS')) \le \sum_{i=1}^K\vol(\pball(Y_i,2r))\leq C(n)Kr^{n+3}\leq C(n,\alpha, \Lambda, \ep)r^{n+3-j}.
\end{align}
Combining \eqref{sub1} and \eqref{sub2}, we obtain that for any $r \in (0,1)$,
\begin{equation} \label{eq:main1x}
\vol \lc \pball_r(\mathcal{S}^{j}_{\epsilon,r}) \bigcap \pball(X,1) \rc\leq C(n,\alpha, \Lambda, \ep) r^{n+3-j}
\end{equation}

Furthermore, since $\MS_{\epsilon}^j \subset \MS_{\epsilon,r}^j$, it follows from \eqref{eq:main1x} that
\begin{equation*}
\vol \lc \pball_r(\mathcal{S}^{j}_{\epsilon}) \bigcap \pball(X,1) \rc\leq C(n,\alpha, \Lambda, \ep) r^{n+3-j}
\end{equation*}
for any $r \in (0,1)$. From this, it is evident that
\begin{align*}
\mathscr H_P^j (\MS_\epsilon^j\bigcap \pball(X,1))\leq C(n,\alpha,\Lambda,\epsilon).
\end{align*}

In summary, this completes the proof of Theorem \ref{volumeestimate1}.

\bibliographystyle{alpha}
\bibliography{Volume}

\vskip10pt

Hanbing Fang, Mathematics Department, Stony Brook University, Stony Brook, NY 11794, United States; Email address: hanbing.fang@stonybrook.edu;\\

Yu Li, Institute of Geometry and Physics, University of Science and Technology of China, No. 96 Jinzhai Road, Hefei, Anhui Province, 230026, China; Hefei National Laboratory, No. 5099 West Wangjiang Road, Hefei, Anhui Province, 230088, China; E-mail: yuli21@ustc.edu.cn. \\

\end{document}